\documentclass[12pt]{amsart}
\usepackage{amssymb}
\usepackage{subfigure}
\usepackage[mathlines]{lineno}
\usepackage[margin=1in, footskip=0.5in]{geometry}
\usepackage{soul}
\usepackage{url}
\usepackage{scrextend}
\usepackage{etoolbox}
\usepackage{caption} 
\usepackage{tikz}
\usepackage[T1]{fontenc}
\usepackage{xcolor}

\usepackage{mathtools}
\usepackage{comment}
\usepackage{theoremref}
\numberwithin{figure}{section} 

\usetikzlibrary{external,quotes,chains}

\newcommand{\abs}[1]{\left\vert#1\right\vert}

\DeclareMathOperator{\Z}{Z}
\DeclareMathOperator{\pt}{pt}
\DeclareMathOperator{\bd}{bd}
\DeclareMathOperator{\ppt}{ppt}
\DeclareMathOperator{\rd}{rd}

\DeclareMathOperator{\comp}{comp}

\DeclareMathOperator{\thr}{th}
\DeclareMathOperator{\PIP}{PIP}
\DeclareMathOperator{\term}{Term}
\DeclareMathOperator{\rev}{Rev}
\DeclareMathOperator{\ct}{ct}
\DeclareMathOperator{\act}{act}

\newcommand{\ptp}{\pt_+}
\newcommand{\Zp}{\Z_+}
\newcommand{\F}{\mathcal{F}}

\newtheorem{theorem}{Theorem}[section]
\newtheorem{lemma}[theorem]{Lemma}
\newtheorem{corollary}[theorem]{Corollary}
\newtheorem{proposition}[theorem]{Proposition}
\newtheorem{observation}[theorem]{Observation}

\theoremstyle{definition}
\newtheorem{definition}[theorem]{Definition}
\newtheorem*{definition*}{Definition}
\newtheorem{example}[theorem]{Example}
\newtheorem{remark}[theorem]{Remark}

\newcommand{\cartProd}{\mathbin{\Box}}

\newcommand{\bit}{\begin{itemize}}
\newcommand{\eit}{\end{itemize}}
\newcommand{\ben}{\begin{enumerate}}
\newcommand{\een}{\end{enumerate}}
\newcommand{\beq}{\begin{equation}}
\newcommand{\eeq}{\end{equation}}
\newcommand{\bea}{\begin{eqnarray*}}
\newcommand{\eea}{\end{eqnarray*}}
\newcommand{\bpf}{\begin{proof}}
\newcommand{\epf}{\end{proof}\ms}
\newcommand{\bmt}{\begin{bmatrix}}
\newcommand{\emt}{\end{bmatrix}}
\newcommand{\ms}{\medskip}

\newcommand{\lc}{\left\lceil}
\newcommand{\rc}{\right\rceil}

\tikzset{vertex/.style={black,fill,draw,minimum size=6pt,inner sep=0pt,circle,thin},bold/.style={black,line width=0.6mm},plain/.style={black,thin},bold edges/.style=bold,plain edges/.style=plain,label distance=1mm,text node/.style={rectangle,fill=none,draw=none},every label/.style=text node,caption node/.style={text node,font=\Large}}

\numberwithin{equation}{section}

\title[New structures for zero forcing and propagation time]{New structures and their applications\\ to variants of zero forcing and propagation time}
\author[L. Hogben]{Leslie Hogben}
\address{Department of Mathematics, Iowa State University, Ames, IA 50011 and American Institute of Mathematics, San Jose, CA 95112}
\email{hogben@aimath.org}
\author[M. Hunnell]{Mark Hunnell}
\address{Department of Mathematics, Winston-Salem State University, Winston-Salem, NC 27110}
\email{hunnellm@wssu.edu}
\author[K. Liu]{Kevin Liu}
\address{Department of Mathematics, University of Washington, Seattle, WA 98195}
\email{kliu15@uw.edu}
\author[H. Schuerger]{Houston Schuerger}
\address{Department of Mathematics, Trinity College, Hartford, CT 06106}
\email{houston.schuerger@trincoll.edu}
\author[B. Small]{Ben Small}
\address{University Place, WA 98466}
\email{bentsm@gmail.com}
\author[Y. Zhang]{Yaqi Zhang} 
\address{Department of Mathematics, Drexel University, Philadelphia, PA 19104}
\email{yaqizhangus@outlook.com}

\begin{document}
\maketitle

\begin{abstract} We introduce a generalization of the concept of a chronological list of forces, called a relaxed chronology.  This concept is used to introduce a new way of formulating the standard zero forcing process, which we refer to as parallel increasing path covers, or PIPs. The combinatorial properties of PIPs are utilized to identify bounds comparing standard zero forcing propagation time to positive semidefinite propagation time.  A collection of paths within a set of PSD forcing trees, called a path bundle, is used to identify the PSD forcing analog of the reversal of a standard zero forcing process, as well as to draw a connection between PSD forcing and rigid-linkage forcing.
\end{abstract}\bigskip

\section{Introduction}\label{s:intro}

Zero forcing is a dynamic coloring process on (finite, simple, and undirected) graphs.  It has multiple applications and has been introduced independently in multiple fields. Examples of these applications include bounding maximum nullity and minimum rank of graphs in combinatorial linear algebra \cite{AIM}, efficient placement of monitors in an electrical grid through power domination in \cite{powerdom}  (with the role of zero forcing evident in \cite{BH}), and the study of control of quantum systems in \cite{graphinfect}.  

In zero forcing and its variants, one starts with a graph whose vertices have all been colored either blue or white, and then an iterative process occurs during which white vertices become blue according to some color change rule.  For most variants, the goal is to identify initial colorings for which the entire graph will eventually become blue after a sufficient number of applications of the associated color change rule. The minimum number of time-steps needed to color the entire graph blue (while forcing all possible vertices at each time-step) is known as the propagation time, and was introduced formally as a graph parameter in \cite{itind} and \cite{proptime}.  The minimum of the sum of the cardinality of a zero forcing set and its propagation time in a graph $G$ is known as the throttling number and was first studied in \cite{thr}. In addition to standard zero forcing, zero forcing variants included in this article are positive semidefinite (PSD) zero forcing introduced in \cite{param}, power domination introduced in \cite{powerdom}, and rigid-linkage forcing introduced in \cite{rigid}.  We will also be interested in the propagation times of PSD forcing introduced in \cite{warnberg} and power domination introduced in \cite{powprop}, as well as the PSD throttling number introduced in \cite{psdthr}. Note that throughout the literature, it is sometimes convenient to consider a chronological list of forces, where exactly one force occurs at a time-step rather than all possible forces.

In this paper, we introduce a generalization of a chronological list of forces, which we call a relaxed chronology. This can be viewed as choosing any subset of possible forces to apply at each time-step in the zero forcing process. For standard zero forcing, this allows us to construct an alternative formulation of the zero forcing process, which we call parallel increasing path covers (or PIPs). PIPs consist of a path cover of a graph and  certain collections of partitions of $\{0,1,2,\dots,K\}$ for some $K\in \mathbb{N}$ (where we use the convention $0\in \mathbb N$), which we call block partitions. These respectively correspond to the forcing chains and active time-steps of vertices with respect to a relaxed chronology. By starting with only a collection of block partitions, PIPs also allow us to construct graphs with predetermined chain sets.  
This formulation has some resemblance to the collections of graphs introduced by Carlson in \cite{throt} to study throttling numbers, but with an emphasis on the neighborhoods of single vertices to allow for the study of local properties.  

Though PIPs correspond to relaxed chronologies in standard zero forcing, they  also have applications to other variants of zero forcing, as well as structural graph properties. Using the specific properties of the block partitions, one can show that the set of active vertices at some time-step in $\{0,1,2,\dots,K\}$ forms a set that is both a PSD forcing set and a power dominating set, as well as a vertex cut when the resulting set of vertices does not consist only of endpoints of the path cover (see Figure \ref{fig:intro} for an example). From these results, we establish upper bounds for PSD propagation time and power domination time in terms of standard propagation time. We also recover well-known results on the terminus of a zero forcing set and  generalizations of reversals  to relaxed chronologies of forces.

\begin{figure}[h!]
    \centering
    \begin{tikzpicture}[scale=1.6,every node/.style=vertex]
\draw[bold edges]
  (0,2) node["$\{0\}$" left]  (v10) {} --
++(1,0) node["$\{1{,}2{,}3\}$" above] (v11) {} --
++(1,0) node["$\{4{,}5\}$" above] (v12) {} --
++(1,0) node["$\{6{,}7{,}8\}$" right] (v13) {};
\draw[bold edges]
  (0,1) node["$\{0{,}1{,}2\}$" left]  (v20) {} --
++(1,0) node["$\{3{,}4\}$" below] (v21) {} --
++(1,0) node["$\{5{,}6{,}7\}$" below] (v22) {} --
++(1,0) node["$\{8\}$" right] (v23) {};
\draw[bold edges]
  (0,0) node["$\{0{,}1{,}2{,}3\}$" left]  (v30) {} --
++(1,0) node["$\{4{,}5\}$" below] (v31) {} --
++(1,0) node["$\{6\}$" below] (v32) {} --
++(1,0) node["$\{7{,}8\}$" right] (v33) {};
\draw[plain edges]
  (v10) -- (v20) -- (v30)
  (v11) -- (v21) -- (v31)
  (v12) -- (v22) -- (v32)
  (v13) -- (v23) -- (v33)
  (v20) -- (v11)
  (v30) -- (v21)
  (v22) -- (v13)
  (v22) -- (v33)
  (v31) -- (v12);
\end{tikzpicture}
    \caption{A graph with path cover shown in bold and vertices labeled by elements in certain block partitions of $\{0,1,2,\ldots,8\}$. Observe that picking the set of vertices whose label contains a fixed $j\in \{1,2,\ldots,7\}$ produces a set that is simultaneously a PSD forcing set, a power dominating set, and a vertex cut of the graph.}
    \label{fig:intro}
\end{figure}
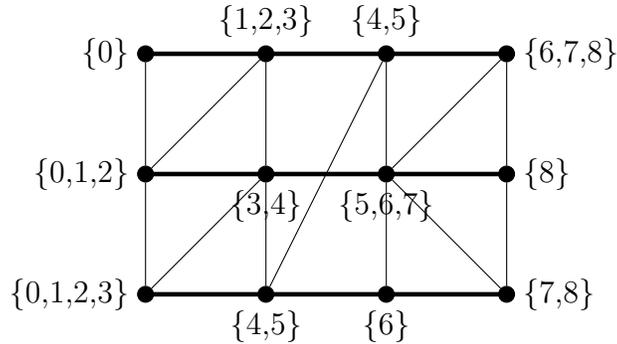

Relaxed chronologies for certain variants of zero forcing also allow for a natural way to restrict forcing to 
subgraphs. We introduce a special case called path bundles, where the restriction of PSD forcing is standard zero forcing. We then use path bundles to construct a lower bound on PSD propagation time using standard zero forcing propagation time.  
We also use path bundles to establish a PSD analog of the reversal of standard zero forcing processes,  and we show that given a set of PSD forcing trees and a fixed vertex $x$, one can find a PSD forcing set of the same cardinality containing $x$ that preserves the PSD forcing trees. Our results on path bundles also have connections to the rigid linkages studied in \cite{rigid}.

We start in Section \ref{s:prelim} with preliminaries. In Section \ref{s:relaxed-chron-PIP}, we define relaxed chronologies and PIPs, and we establish a correspondence between PIPs and relaxed chronologies for standard zero forcing. In Section \ref{s:PIP-apps}, we apply our results on PIPs to establish results for several variants of zero forcing, such as those illustrated in Figure \ref{fig:intro}. Finally, in Section \ref{s:relaxed-chron-subgraphs}, we consider restrictions of relaxed chronologies to subgraphs and establish results on path bundles. Our results on power domination and rigid linkages are included in appendices. 


\section{Preliminaries}\label{s:prelim}

In this section we provide definitions for graphs, zero forcing and variants, and propagation.  Additional background can be found in \cite[Part 4]{HLS22}.

\subsection{Graph terminology}
A {\em graph} $G$ is a pair $(V(G),E(G))$, where $V(G)$ is the set of {\em vertices} and $E(G)$ is a set of 2-element sets of vertices called {\em edges}.  To help differentiate between subsets of $V(G)$ of cardinality two and edges, given two distinct vertices $u,v\in V(G)$, the subset of $V(G)$ composed of these two vertices will be denoted $\{u,v\}$ while the edge between $u$ and $v$ will be denoted $uv$ (or $vu$).  All graphs are assumed to be finite and {\em simple}, that is, to have neither loops (edges between between a vertex and itself) nor multiple edges between any two distinct vertices.  

If $G$ and $H$ are graphs such that $V(H) \subseteq V(G)$ and $E(H) \subseteq E(G)$, then $H$ is a {\em subgraph} of $G$.  If $H$ is a subgraph of $G$ and for any  two vertices $u,v \in V(H)$ we have $uv \in E(H)$ if and only if $uv \in E(G)$, then $H$ is an {\em induced subgraph} of $G$.  Given a subset of vertices $S \subseteq V(G)$, the induced subgraph $H$ of $G$ with vertex set $V(H)=S$ will be denoted $G[S]$.  In addition, for a given set of vertices $B \subseteq V(G)$, the notation $G-B$ will be used to denote $G[V(G) \setminus B]$.  

Given a vertex $v \in V(G)$, the {\em open neighborhood} of $v$ in $G$, denoted by $N_G(v)$, is the collection of vertices $u$ such that $vu \in E(G)$.  The {\em closed neighborhood} of a vertex $v$, denoted by $N_G[v]$, is $N_G[v]=N_G(v) \cup \{v\}$.  Likewise, given a set of vertices $S \subseteq V(G)$, the open (respectively, closed) neighborhood of $S$ is defined to be the union of the open  (respectively, closed) neighborhoods of the vertices of $S$.   

A \emph{path} is a sequence of distinct vertices $v_1,v_2,\dots,v_m$ such that for each $i$ with $1 \leq i \leq m-1$ we have $v_iv_{i+1} \in E(G)$.  
Given a pair of vertices $u,v \in V(G)$, a {\em $uv$-path} is a path $v_1,v_2,\dots,v_m$ such that $u=v_1$, $v=v_m$.  A graph $G$ is said to be \emph{connected} if for each pair of vertices $u,v \in V(G)$ there exists a $uv$-path.  A graph is {\em disconnected} if it is not connected.   The maximal connected subgraphs of $G$ are its {\em components}.  The set of these components is denoted $\comp(G)$, and given a vertex $v$, the component of $G$ that contains the vertex $v$ is denoted $\comp(G,v)$.  If $S \subset V(G)$ such that $G-S$ has more components than $G$ does, then $S$ is a \emph{vertex cut} of $G$.
We also view a path as a graph: the \emph{path} $P_n$ is the graph with $V(P_n)=\{v_1,\dots,v_n\}$ and $E(P_n)=\{{ v_iv_{i+1}}: i=1,\dots,n-1\}$.    A \emph{path cover} of $G$ is a collection of induced paths in $G$ with the property that every vertex of $G$ is in exactly one path.

\subsection{Zero forcing}
In standard zero forcing and each of the variants discussed here, one starts with a graph and colors every vertex of the graph either blue or white.  Given a set  of blue vertices (with the other vertices colored white), a process is started during which blue vertices cause white vertices to become blue.  Each variant is defined by its color change rule, which governs under what circumstances a blue vertex can cause a white vertex to become blue during this process. When a blue vertex $u$ causes a white vertex $v$ to become blue, this is referred to as \emph{forcing} and denoted $u \rightarrow v$ (there is often a choice as to which vertex is chosen to force $v$ among those that can).   It is worth noting that once a vertex is blue, it will remain blue.   In most of the variants discussed in this paper, the goal is for every vertex in the graph to become blue, so the color change rule will be applied until no further forces are possible.  The relevant color change rules are as follows.  
\begin{itemize}
    \item Standard zero forcing color change rule (CCR-$\Z$): If a blue vertex $u$ has a unique white neighbor $v$, then $u$ can force $v$ to become blue.
    \item Positive semidefinite (PSD) zero forcing color change rule (CCR-$\Zp$): If $B$ is the set of currently blue vertices, $C$ is a component of $G-B$, and $u$ is a blue vertex such that $N_G(u) \cap V(C)=\{v\}$, then $u$ can force $v$ to become blue.
\end{itemize}
 Power domination on graphs was introduced in \cite{powerdom} before zero forcing was defined as a separate parameter. Following the introduction of zero forcing in \cite{AIM} and \cite{graphinfect} and the work of Brueni and Heath simplifying power domination in \cite{BH}, power domination can be viewed as the process of starting with a set of initially blue vertices $B$, coloring $ N_G[B]$ blue during the first step, and for subsequent steps (that is, steps $k\geq 2$) applying CCR-$\Z$. We will also be discussing rigid-linkage forcing (RL-forcing) in Appendix \ref{appendix:linkage}, but we defer the definition of its color change rule to that section.

A {\em standard zero forcing set} of a graph $G$ is a set of vertices $B$ such that if $B$ is the set of initially blue vertices and CCR-$\Z$ is applied a sufficient number of times, then all vertices of $G$ become blue.  The {\em standard zero forcing number} of a graph $G$ is $\Z(G)=\min\{\abs{B}: B \text{ is a standard zero forcing set of }G\}$. A {\em standard minimum zero forcing set} $B$ of a graph $G$ is a standard zero forcing set of $G$ such that $\abs{B}=\Z(G)$. 
A {\em (minimum) PSD forcing set} and the {\em PSD forcing number $\Zp(G)$} of a graph are defined analogously using CCR-$\Zp$ as the color change rule. 
The terms  {\em (minimum) power dominating set} and {\em power domination number} are defined analogously by applying CCR-$\Z$ to $N_G[B]$ where $B$ is the set of initially blue vertices.  

\subsection{Forcing and propagation}

The process of coloring vertices blue has been viewed from various perspectives, including performing   only one force at a time or as many forces as  are independently  possible in each step.  In a zero 
forcing process applied to a zero
forcing set $B$  where exactly one white vertex is forced  blue during each step, a list of these forces in the order in which they occur is known as a {\em chronological list of forces}. Note that a chronological list of forces of $B$ contains $K=|V(G)|-|B|$ forces.  Let $F^{(k)}$ be the set containing the one force $u \rightarrow v$ that occurs during step $k$ of the process.  Then this chronological list of forces can also be viewed as an ordered set and denoted by $\mathcal F=\{F^{(k)}\}_{k=1}^K$.   There are also times when considering a set of forces without reference to a specific order may be useful and in such instances we will denote the set of forces by $F$.
Given a set of initially blue vertices, a great deal of choice may occur in creating a chronological list of forces that colors all vertices blue, because at each step of the zero forcing process there may be multiple vertices capable of being forced and multiple vertices capable of forcing each such vertex.  However, it is well-known that the order of forces does not affect what vertices can be colored blue by a given set of initially blue vertices \cite{AIM}.

In a \emph{propagation process}, at each step the color change rule is applied independently to each vertex  that is blue when the step begins and then the set of blue vertices is updated to include the set of vertices that have been forced blue; this is referred to as a \emph{time-step}.   More formally, for any fixed color change rule, for any graph $G$, and for any initial set $B$ of blue vertices, let $B^{[0]}=B^{(0)}=B$.  For $k\ge 1$, define $B^{(k)}$ to be the set of vertices that can be forced blue during time-step $k$ of propagation, i.e., $B^{(k)}=\{w: 
\mbox{$w$ can be forced by some $v$, given the set of blue vertices is $B^{[k-1]}$}\}.$
Define $B^{[k]}$ to be the set of vertices that are blue after time-step $k$, i.e., $B^{[k]}=B^{[k-1]}\cup B^{(k)}$.  
Suppose $B$ is a zero 
forcing set.  Define the \emph{round function} by $\rd(v)=k$ where $k$ is the unique index such that $v\in B^{(k)}$.  Consider the propagation process for $B$, i.e., the sequence of blue sets $B^{(k)}$. Let $t$ be the least $k$ such that $B^{[k]}=V(G)$.  For $k= 1,\dots,t$ and for each vertex $u$ in $B^{(k)}$, choose a vertex $v_u$ such that $v_u\in B^{[k-1]}$ and $v_u$ can force $u$ (in accordance with the color change rule being used). Let $F^{(k)}=\{v_u\to u:\rd(u)=k\}$.  The ordered set $\F$ of sets $F^{(k)}$ is a  \emph{propagating family of forces}  and $F=\bigcup_{k=1}^t F^{(k)}$ is a \emph{propagating set of forces}. For a given propagating set of forces, a vertex $v$ is \emph{active} at time-step $k$ if $v$ is blue after time-step $k$ but $v$ has not yet performed a force.  Note that in a propagation process, the sets $B^{(k)}$ and $B^{[k]}$ are uniquely determined by $B$. 
Furthermore, each of the sets $B^{(k)}$ and $F^{(k)}$ is nonempty for $k=1,\dots,t$.

Consider the standard zero forcing color change rule.  For a set of forces $F$ of  a zero 
forcing set $B$ of $G$, each vertex $v\in B$ defines a  {\em forcing chain} $C_{v}=(v=v_0,v_1,\dots,v_k)$ where $v_{i-1}\to v_i\in F$ for $i=1,\dots k$ and $v_k$ does not perform a force. Every vertex of $G$ appears in exactly one forcing chain defined from $F$ (note $k=0$ is allowed with the forcing chain $(v_0)$).  The {\em chain set} defined by $F$ is the collection of forcing chains $\mathcal C=\{C_v\}_{v \in B}$.  The subgraph of $G$ induced by the vertices of a forcing chain is a path.  For a set of forces {$F$} of $B$ using the PSD forcing color change rule, similar sets of vertices are constructed by $F$, but in this case each vertex might force multiple vertices, so rather than forming induced paths the process constructs induced trees.  For this reason, a set of forces $F$ using the PSD forcing color change rule (or a set of PSD forces) constructs forcing trees: For a vertex $b$ in a PSD forcing set $B$ and a set of PSD forces $\F$ of $B$, define $V_b$ to be the set of all vertices  $w$ such that  there is a sequence of forces  $b=v_1\to v_2\to\dots\to v_k=w$ in $\F$  (the empty sequence of forces is permitted, i.e., $b\in V_b$).
The \emph{forcing tree $T_b$}  is the induced subgraph $T_b=G[V_b]$.
The \emph{forcing tree cover}  (for a set of forces $\F$) is $\mathcal{T}=\{T_b : b\in B\}$ and every vertex of $G$ is a vertex of some tree in  the forcing tree cover.  Note that each vertex in a PSD forcing set is the first vertex of a forcing tree, so the number of forcing trees is the cardinality of the PSD forcing set \cite{ekstrand}. 

While we will be considering several color change rules, our focus on propagation time will be restricted to standard zero forcing and PSD forcing.  
Let $B$ be a standard zero forcing set  of a graph $G$.   The {\em standard propagation time of $B$} is $\pt(G,B)=t$ where $t$ is the least $k$ such that $B^{[k]}=V(G)$.  The {\em standard $k$-propagation time} of $G$ is
\[\pt(G,k)=\min\{\pt(G,B): B\text{ is a standard zero forcing set of } G \text{ and }\abs{B}=k\}.\]
Finally, the {\em standard propagation time} of a graph $G$ is
\[\pt(G)=\min\{\pt(G,B): B\text{ is a standard zero forcing set of } G \text{ and }\abs{B}=\Z(G)\}.\]
A set of vertices $B \subseteq V(G)$ is said to be an {\em efficient standard zero forcing set} if $B$ is a standard zero forcing set such that $\abs{B}=\Z(G)$ and $\pt(G,B)=\pt(G)$.  Similarly $B$ is said to be a {\em $k$-efficient standard zero forcing set} if $B$ is a standard zero forcing set such that $\abs{B}=k$ and $\pt(G,B)=\pt(G,k)$.  The {\em standard throttling number} of a graph $G$ is
\[\thr(G)=\min\{\abs{B}+\pt(G,B): B \text{ is a standard zero forcing set of }G\}.\]
For a  PSD forcing set $B$ of a graph $G$, the {\em PSD propagation time of $B$}, denoted by  $\pt_+(G,B)$, the {\em PSD  $k$-propagation time} of $G$, denoted by   $\pt_+(G,k)$, the {\em PSD propagation time} of $G$, denoted by $\pt_+(G)$, and the {\em PSD throttling time}, denoted by $\thr_+(G)$, are defined analogously; and for a power dominating set $B$ of a graph $G$, the {\em power propagation time of $B$}, denoted by  $\ppt(G,B)$, the {\em power $k$-propagation time} of $G$, denoted by   $\ppt(G,k)$, and the {\em power propagation time} of $G$, denoted by $\ppt(G)$, are defined analogously.


\section{Relaxed chronologies and parallel increasing path covers}\label{s:relaxed-chron-PIP}

 In this section we define two related new structures: relaxed chronologies, which generalize both chronological lists of forces and propagating families of forces,  and parallel increasing path covers, which model chain sets and carry additional information.

\subsection{Relaxed chronologies}

We start by presenting a flexible framework for discussing zero forcing, called a relaxed chronology.  To help introduce and motivate this framework, we offer the next example.

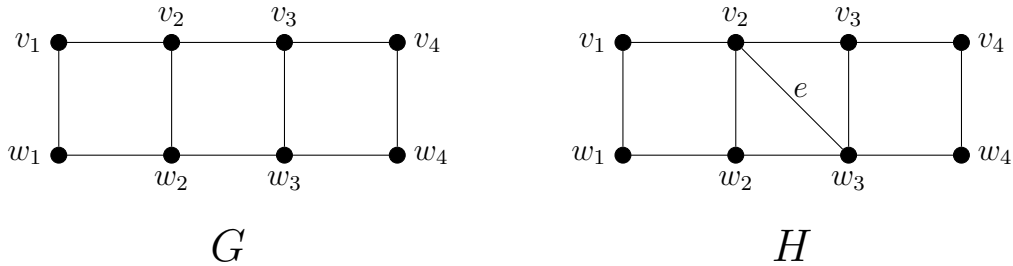
\begin{figure}[h]
\begin{tikzpicture}[scale=1.5,every node/.style=vertex,plain edges]
\begin{scope}[name prefix=G]
\draw
  (0,1) node["$v_1$" left]  (v1) {} --
++(1,0) node["$v_2$" above] (v2) {} --
++(1,0) node["$v_3$" above] (v3) {} --
++(1,0) node["$v_4$" right] (v4) {};
\draw
  (0,0) node["$w_1$" left]  (w1) {} --
++(1,0) node["$w_2$" below] (w2) {} --
++(1,0) node["$w_3$" below] (w3) {} --
++(1,0) node["$w_4$" right] (w4) {};
\draw
  (v1) -- (w1) (v2) -- (w2) (v3) -- (w3) (v4) -- (w4);
\node[caption node] at (1.5,-0.8) {$G$};
\end{scope}
\begin{scope}[xshift=5cm,name prefix=H]
\draw
  (0,1) node["$v_1$" left]  (v1) {} --
++(1,0) node["$v_2$" above] (v2) {} --
++(1,0) node["$v_3$" above] (v3) {} --
++(1,0) node["$v_4$" right] (v4) {};
\draw
  (0,0) node["$w_1$" left]  (w1) {} --
++(1,0) node["$w_2$" below] (w2) {} --
++(1,0) node["$w_3$" below] (w3) {} --
++(1,0) node["$w_4$" right] (w4) {};
\draw
  (v1) -- (w1) (v2) -- (w2) (v3) -- (w3) (v4) -- (w4)
  (v2) to["$e$" text node] (w3);
\node[caption node] at (1.5,-0.8) {$H$};
\end{scope}
\end{tikzpicture}
\caption{A zero forcing example.}
\label{illustrate-chron}
\end{figure}

\begin{example}
\label{chron-example}
Consider the graphs $G$ and $H$ in Figure \ref{illustrate-chron}, which are identical except that $e=v_2w_3\in E(H)$.  $B=\{v_1,w_1\}$ is a standard zero forcing set for both $G$ and $H$, and the ordered set $(v_1\to v_2,w_1\to w_2,w_2\to w_3,v_2\to v_3,v_3\to v_4,w_3\to w_4)$ is a chronological list of forces for both graphs.  Note that $\mathcal{F}=(F^{(1)}=\{v_1\to v_2,w_1\to w_2\},F^{(2)}=\{w_2\to w_3\},F^{(3)}=\{v_2\to v_3\},F^{(4)}=\{v_3\to v_4,w_3\to w_4\})$ is a propagating family of forces on $H$.  We can also view $\F$ as an ordered set of sets of forces in $G$ with the forces in $F^{(k)}$ all occurring in the $k$-th time-step.  We particularly note that performing those forces at the specified time-steps is consistent with the standard zero forcing color change rule. Indeed, this would hold if $\mathcal{F}$ were any other propagating family of forces on $H$ not involving a force along $e$.
\end{example}

It is clear that neither a chronological list of forces nor a propagating family of forces describes the process of $\F$ forcing in $G$ in the preceding example, since $\F$ is neither one with respect to $G$.  In order to discuss the observations of Example \ref{chron-example} formally, we therefore introduce a more general forcing process.  As we will see later on, this new type of forcing process is useful in both the standard zero forcing setting and the PSD forcing setting, so we define the process universally.

\begin{definition}
Let $G$ be a graph, let $X$-CCR be a color change rule, and let $B$ be any set of blue vertices. The \emph{collection of possible $X$-forces} for $B$ in $G$ is
\[S_X(G,B)=\{v\to u: v\to u \text{ is a valid $X$-force  in $G$ when exactly the vertices in $B$ are blue}\}.\]
\end{definition}

\begin{definition}
\label{rel_chron}
Let $G$ be a graph, let $X$-CCR be a color change rule, and let $B$ be an $X$-forcing set of $G$. Consider an ordered family  $\mathcal{F}=\{F^{(k)}\}_{k=1}^K$, where each $F^{(k)}$ is a set of $X$-forces. 
Define \[E_{\mathcal{F}}^{[0]}=E_{\F}^{(0)}=B, \ E_{\mathcal{F}}^{(k)}=\{u:v\to u \in F^{(k)} \text{ for some $v$}\}, \mbox{ and } E_{\mathcal{F}}^{[k]}=\bigcup_{j=0}^k E_{\mathcal{F}}^{(j)}\] for $k=1,2,\dots,K$. Then $\mathcal{F}$ is a \emph{relaxed chronology} of $X$-forces for $B$ on $G$ if 
\begin{enumerate}
    \item $F^{(k)}\subseteq S_X(G,E_{\mathcal{F}}^{[k-1]})$ for $k=1,2,\dots,K$,
    \item $v_1\to u,v_2\to u\in F^{(k)}$ implies $v_1=v_2$, and \label{rel_chron_cond2}
    \item $E_{\mathcal{F}}^{[K]}=V(G)$. \label{rel_chron_cond3}
\end{enumerate}
In this case, we call $\{E_{\mathcal{F}}^{[k]}\}_{k=0}^K$ the \emph{expansion sequence} of $B$ induced by $\mathcal{F}$, each individual $E_{\mathcal{F}}^{[k]}$ the \emph{$k$-th expansion} of $B$ induced by $\mathcal{F}$, and $K$ the \emph{completion time} of $\mathcal{F}$, which we  denote by $\ct(\mathcal{F})$. When the color change rule is clear from context, the $X$ can be dropped.
\end{definition}

Notice that given an $X$-forcing set, we can inductively construct a relaxed chronology by choosing a subset of the possible $X$-forces at each time-step. Condition \ref{rel_chron_cond2} of Definition \ref{rel_chron} ensures that multiple vertices do not force the same vertex, while condition \ref{rel_chron_cond3} ensures that we do not terminate until some time after all vertices of $G$ are blue. A particularly unusual feature of this process is that we will allow for $F^{(k)}$ to be the empty set since we allow any subset $F^{(k)}$ of the valid $X$-forces at each time-step to be performed. Hence, $\ct(\mathcal{F})$ is not necessarily the first time-step when all of $V(G)$ is blue.

In the special case of the standard zero forcing color change rule, we define analogs of two definitions from zero forcing.

\begin{definition}\label{def:F-active-chain}
Let $\mathcal{F}=\{F^{(k)}\}_{k=1}^K$ be a relaxed chronology of forces for the standard zero forcing color change rule on a graph $G$. 
\begin{enumerate}
    \item For $v\in V(G)$, the \emph{$\mathcal{F}$-active times} $\act_{\mathcal{F}}(v)\subseteq \{0,1,2,\dots,K\}$ are the time-steps when $v$ is active with respect to $\mathcal{F}$, that is, $k\in \{0,1,2,\dots, K\}$ is in $\act_{\mathcal{F}}(v)$ if and only if $v$ is blue after time-step $k$ and has not performed a force (the  $\F$ can be omitted when it is clear from context).  
    \item The \emph{chain set} defined by $\mathcal{F}$ is the chain set of the underlying set of forces $\bigcup_{k=1}^K F^{(k)}$.
\end{enumerate}
\end{definition}

Observe that relaxed chronologies generalize both propagating families of forces and chronological lists of forces. If $F^{(k)}$ is maximal for all $k$, then the expansion sequence $\{E_{\mathcal{F}}^{[k]}\}$ reduces to  $\{B^{[k]}\}$, the set of vertices that are blue after time-step $k$ using the customary $X$-propagation process. Additionally, if $|F^{(k)}|=1$ for all $k$, then $\mathcal{F}$ reduces to a chronological list of forces.

Returning briefly to Example \ref{chron-example}, observe that $\F$ is a relaxed chronology for $B$ in both $G$ and $H$. More generally, any relaxed chronology for $B$ in $H$ that does not contain a force along $e$ is a relaxed chronology for $B$ in $G$ since $G=H-e$.

\subsection{Parallel increasing path covers}
Using the framework of relaxed chronologies, we can now develop a model of chain sets from a global perspective.  The key objects in this new perspective are structures called parallel increasing path covers.  To help motivate and introduce this model and these structures, we consider another example.

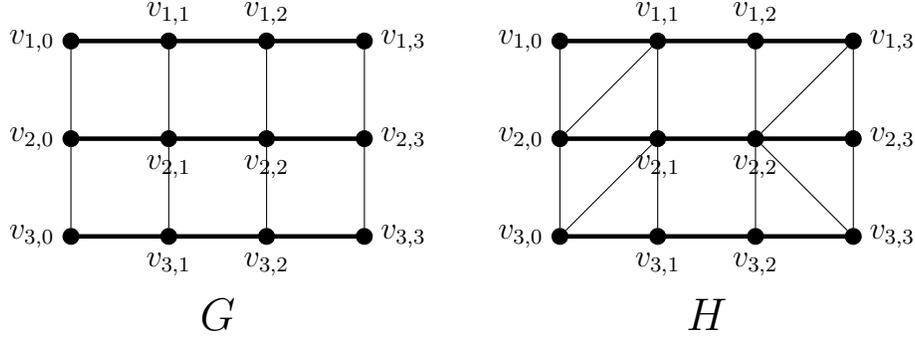
\begin{figure}[h]
\begin{tikzpicture}[scale=1.3,every node/.style=vertex]
\begin{scope}[name prefix=G]
\draw[bold edges]
  (0,2) node["$v_{1,0}$" left]  (v10) {} --
++(1,0) node["$v_{1,1}$" above] (v11) {} --
++(1,0) node["$v_{1,2}$" above] (v12) {} --
++(1,0) node["$v_{1,3}$" right] (v13) {};
\draw[bold edges]
  (0,1) node["$v_{2,0}$" left]  (v20) {} --
++(1,0) node["$v_{2,1}$" below] (v21) {} --
++(1,0) node["$v_{2,2}$" below] (v22) {} --
++(1,0) node["$v_{2,3}$" right] (v23) {};
\draw[bold edges]
  (0,0) node["$v_{3,0}$" left]  (v30) {} --
++(1,0) node["$v_{3,1}$" below] (v31) {} --
++(1,0) node["$v_{3,2}$" below] (v32) {} --
++(1,0) node["$v_{3,3}$" right] (v33) {};
\draw[plain edges]
  (v10) -- (v20) -- (v30)
  (v11) -- (v21) -- (v31)
  (v12) -- (v22) -- (v32)
  (v13) -- (v23) -- (v33);
\node[caption node] at (1.5,-0.8) {$G$};
\end{scope}
\begin{scope}[xshift=5cm,name prefix=H]
\draw[bold edges]
  (0,2) node["$v_{1,0}$" left]  (v10) {} --
++(1,0) node["$v_{1,1}$" above] (v11) {} --
++(1,0) node["$v_{1,2}$" above] (v12) {} --
++(1,0) node["$v_{1,3}$" right] (v13) {};
\draw[bold edges]
  (0,1) node["$v_{2,0}$" left]  (v20) {} --
++(1,0) node["$v_{2,1}$" below] (v21) {} --
++(1,0) node["$v_{2,2}$" below] (v22) {} --
++(1,0) node["$v_{2,3}$" right] (v23) {};
\draw[bold edges]
  (0,0) node["$v_{3,0}$" left]  (v30) {} --
++(1,0) node["$v_{3,1}$" below] (v31) {} --
++(1,0) node["$v_{3,2}$" below] (v32) {} --
++(1,0) node["$v_{3,3}$" right] (v33) {};
\draw[plain edges]
  (v10) -- (v20) -- (v30)
  (v11) -- (v21) -- (v31)
  (v12) -- (v22) -- (v32)
  (v13) -- (v23) -- (v33)
  (v20) -- (v11)
  (v30) -- (v21)
  (v22) -- (v13)
  (v22) -- (v33);
\node[caption node] at (1.5,-0.8) {$H$};
\end{scope}
\end{tikzpicture}
\caption{Another zero forcing example.}
\label{motivate-pip}
\end{figure}

\begin{example} \label{ex:motivate-pip}
Consider the graphs in Figure \ref{motivate-pip}.  $B=\{v_{1,0},v_{2,0},v_{3,0}\}$ is a standard zero forcing set for each, and $\F=\{F^{(k)}\}_{k=1}^8=(\{v_{1,0}\to v_{1,1}\},\emptyset,\{v_{2,0}\to v_{2,1}\},\{v_{3,0}\to v_{3,1},v_{1,1}\to v_{1,2}\},\{v_{2,1}\to v_{2,2}\},\{v_{1,2}\to v_{1,3},v_{3,1}\to v_{3,2}\},\{v_{3,2}\to v_{3,3}\},\{v_{2,2}\to v_{2,3}\})$ is a relaxed chronology for both graphs.  For each vertex $v_{i,j}$, let $A_{i,j}=\act_{\mathcal{F}}(v_{i,j})$.
Then we obtain the following:
\begin{align*}
    A_{1,0}&=\{0\} & A_{1,1}&=\{1,2,3\} & A_{1,2}&=\{4,5\} & A_{1,3}&=\{6,7,8\} \\
    A_{2,0}&=\{0,1,2\} & A_{2,1}&=\{3,4\} & A_{2,2}&=\{5,6,7\} & A_{2,3}&=\{8\} \\
    A_{3,0}&=\{0,1,2,3\} & A_{3,1}&=\{4,5\} & A_{3,2}&=\{6\} & A_{3,3}&=\{7,8\}.
\end{align*}

\noindent From these sets, we observe some interesting properties, which hold for both $G$ and $H$:
\begin{enumerate}
    \item Only one vertex per forcing chain is active after a given time-step, so for each $i \in \{1,2,3\}$, the collection of sets $\{A_{i,j}\}_{j=0}^3$ partitions the set $\{0,1,2,\dots,8\}$.
    \item If two vertices $v_{i_1,j_1}$ and $v_{i_2,j_2}$ are adjacent but are not contained in the same forcing chain, then
\[A_{i_1,j_1} \cap A_{i_2,j_2} \not = \emptyset.\]
\end{enumerate}
\end{example}

To further explore these two properties we provide the following definitions.
 Note that if $S_1$ and $S_2$ are sets of integers such that $x_1<x_2$ for all $x_1\in S_1$ and $x_2\in S_2$, then we write $S_1<S_2$.
\begin{definition}
Let $K \in \mathbb N$, and let $\mathcal A=\{A_j\}_{j=0}^{n_{\mathcal A}-1}$ be a partition of the set $\{0,1,2,\dots,K\}$ into $n_{\mathcal{A}}$ parts. 
 If $j_1 < j_2$ implies that   $ A_{j_1}< A_{j_2}$ for each pair $j_1,j_2 \in \{0,1,2,\dots,n_{\mathcal A}-1\}$, 
then we say that $\mathcal A=(A_j)_{j=0}^{n_{\mathcal A}-1}$ is a \emph{block partition} of $\{0,1,2,\dots,K\}$. 
\end{definition}

\begin{definition}\label{def:PIP}
Let $G$ be a graph, and let $\mathcal Q=\{Q_i\}_{i=1}^m$ be a path cover of $G$ with $n_i=\abs{Q_i}$ for each $i \in \{1,2,\dots,m\}$. Label $V(G)$ so that the vertices of the path $Q_i$ are $\{v_{i,j}\}_{j=0}^{n_i-1}$ in path order (i.e., $v_{i,j_1}v_{i,j_2} \in E(G)$ if and only if $\abs{j_1-j_2}=1$).  Choose $K \in \mathbb N$
and for each $i \in \{1,2,\dots,m\}$, let $\mathcal A_i=(A_{i,j})_{j=0}^{n_i-1}$ be a block partition of $\{0,1,2,\dots,K\}$, where we say $A_{i,j}$ \emph{corresponds} to vertex $v_{i,j}$.  If  
  for distinct $i_1,i_2 \in \{1,2,\dots,m\}$, 
  \begin{equation}\label{non-empty-int}
      v_{i_1,j_1}v_{i_2,j_2} \in E(G) \mbox{ implies }A_{i_1,j_1} \cap A_{i_2,j_2} \not = \emptyset,
  \end{equation}
then $\mathcal Q$ is a \emph{parallel increasing path cover} (or \emph{PIP}) of $G$, and the collection of block partitions $\{\mathcal{A}_{i}\}_{i=1}^m$ is a {\em witness} that $\mathcal Q$ is a parallel increasing path cover  of $G$. Define $\PIP(G)$ to be
\[\PIP(G)=\min\big{\{}\abs{\mathcal Q}: \mathcal Q\text{~is a parallel increasing path cover of~} G\big{\}}.\]
If $\mathcal Q$ is a parallel increasing path cover of $G$ such that $\abs{\mathcal Q}=\PIP(G),$ then we will say that $\mathcal Q$ is a \emph{minimum parallel increasing path cover} of $G$.
\end{definition}

Note that for any path cover $\mathcal Q=\{Q_i\}_{i=1}^m$, one can always choose sufficiently large $K\in \mathbb{N}$ and define a collection of block partitions $\{\mathcal{A}_i\}_{i=1}^m$ of $\{0,1,2,\dots,K\}$ where each $A_{i,j}$ corresponds to a vertex $v_{i,j}$ (which is the $(j+1)$th vertex in the $i$th path), but the key property of parallel increasing path covers is that a collection of block partitions can be chosen that satisfies the property in (\ref{non-empty-int}).

For the graphs in Example \ref{ex:motivate-pip}, the bold edges form a path cover $\mathcal Q$, and the block partitions $\mathcal A_i=(A_{i,j})_{j=0}^3$ witness that $\mathcal Q$ is a parallel increasing path cover of those graphs.  As another example, for the tree in Figure \ref{pip-tree}(a), a path cover $\mathcal Q$ is indicated by thick edges. The following collection of block partitions $\{\mathcal A_i\}_{i=1}^3$ is a witness that $\mathcal Q$ is a parallel increasing path cover of that tree:

\begin{center}
\begin{tabular}{ c c c c c }
$\mathcal{A}_1:$ & $A_{1,0}=\{0\}$ & $A_{1,1}=\{1\}$ & $A_{1,2}=\{2,3,4\}$\\
$\mathcal{A}_2:$ & $A_{2,0}=\{0,1\}$ & $A_{2,1}=\{2\}$ & $A_{2,2}=\{3,4\}$\\
$\mathcal{A}_3:$ & $A_{3,0}=\{0\}$ & $A_{3,1}=\{1,2\}$ & $A_{3,2}=\{3\}$ & $A_{3,3}=\{4\}$.
\end{tabular}
\end{center}

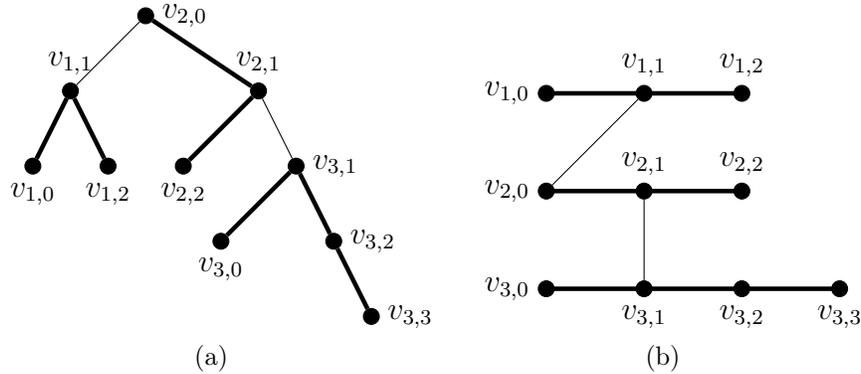
\begin{figure}[h!]
    \centering
    \subfigure[]{
\begin{tikzpicture}[every node/.style=vertex,plain edges]
\node["$v_{1,0}$" below] (v10) at (-1.5,-2) {};
\node["$v_{1,1}$" above] (v11) at (-1,-1) {} edge[bold] (v10);
\node["$v_{1,2}$" below] (v12) at (-0.5,-2) {} edge[bold] (v11);
\node["$v_{2,0}$" right] (v20) at (0,0) {} edge (v11);
\node["$v_{2,1}$" above] (v21) at (1.5,-1) {} edge[bold] (v20);
\node["$v_{2,2}$" below] (v22) at (0.5,-2) {} edge[bold] (v21);
\node["$v_{3,0}$" below] (v30) at (1,-3) {};
\node["$v_{3,1}$" right] (v31) at (2,-2) {} edge[bold] (v30) edge (v21);
\node["$v_{3,2}$" right] (v32) at (2.5,-3) {} edge[bold] (v31);
\node["$v_{3,3}$" right] (v33) at (3,-4) {} edge[bold] (v32);
\end{tikzpicture}}
\quad
\subfigure[]{
\begin{tikzpicture}[scale=1.3,every node/.style=vertex]
\draw[bold edges]
  (0,2) node["$v_{1,0}$" left] {} --
++(1,0) node["$v_{1,1}$" above] (v11) {} --
++(1,0) node["$v_{1,2}$" above] {};
\draw[bold edges]
  (0,1) node["$v_{2,0}$" left]  (v20) {} --
++(1,0) node["$v_{2,1}$" above] (v21) {} --
++(1,0) node["$v_{2,2}$" above] {};
\draw[bold edges]
  (0,0) node["$v_{3,0}$" left] {} --
++(1,0) node["$v_{3,1}$" below] (v31) {} --
++(1,0) node["$v_{3,2}$" below] {} --
++(1,0) node["$v_{3,3}$" below] {};
\draw[plain edges]
  (v11) -- (v20)
  (v21) -- (v31);
\end{tikzpicture}}
    \caption{\label{pip-tree}A parallel increasing path cover shown (a) with the tree drawn naturally, and (b) redrawn with the paths horizontal.}
\end{figure}

The name parallel increasing path cover has been chosen because given a parallel increasing path cover $\mathcal Q$, for each $Q_i, Q_j \in \mathcal Q$ distinct, $G\left[V(Q_i) \cup V(Q_j)\right]$ has the familiar structure of a graph on two parallel paths.  However, even though the structure of two parallel paths might be very familiar and thus be a desirable property for a definition, a collection of paths which taken pairwise are parallel paths does not necessarily form a parallel increasing path cover, as seen in Figure \ref{pairwise}.  

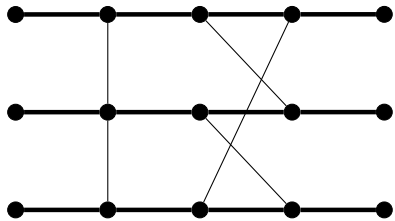
\begin{figure}[h!]
    \centering
\begin{tikzpicture}[every node/.style={vertex,on chain,join},every join/.style=bold,scale=1.3]
\path[start chain=1] (0,2) node {} node {} node {} node {} node {};
\path[start chain=2] (0,1)
  node {}
  node[join=with 1-2 by plain] {}
  node {}
  node[join=with 1-3 by plain] {}
  node {};
\path[start chain=3] (0,0)
  node {}
  node[join=with 2-2 by plain] {}
  node[join=with 1-4 by plain] {}
  node[join=with 2-3 by plain] {}
  node {};
\end{tikzpicture}
    \caption{A collection of paths that is not a parallel increasing path cover, but where the vertices of any two paths induce a graph on two parallel paths.}
    \label{pairwise}
\end{figure}

Rather, an additional property is needed,
which is the requirement that if two elements of distinct block partitions correspond to two adjacent vertices, then these elements must have a nonempty intersection, as stated in (\ref{non-empty-int}).  This property ensures that in some sense the paths progress in the same direction and in a manner consistent with standard zero forcing. 

\subsection{The relationship between relaxed chronologies, standard zero forcing, and parallel increasing path covers} The next theorem explicitly establishes the connection between parallel increasing path covers and relaxed chronologies in standard zero forcing. 

\begin{theorem}\label{parallel}
Let $G$ be a graph. Then $\mathcal Q$ is a parallel increasing path cover of $G$ if and only if $\mathcal Q$ is a chain set for some relaxed chronology $\mathcal F$ of a standard zero forcing set $B$ of $G$.  Under this correspondence, a collection of block partitions  $\{(A_{i,j})_{j=0}^{n_i-1}\}_{i=1}^m$ that is a witness of $\mathcal Q$ as a PIP corresponds to a relaxed chronology $\mathcal{F}$ where $\{(A_{i,j})_{j=0}^{n_i-1}\}_{i=1}^m$ records the active time-steps for vertices  $\{v_{i,j}\}_{j=0}^{n_i-1}\,_{i=1}^m$ with $\{v_{i,0}\}_{i=1}^m$ as the zero forcing set. \end{theorem}

\begin{proof}
Let $\mathcal Q$ be a parallel increasing path cover of $G$, with  the vertices labeled $\{v_{i,j}\}_{j=0}^{n_i-1}\,_{i=1}^m$ as in Definition \ref{def:PIP} and  $\{(A_{i,j})_{j=0}^{n_i-1}\}_{i=1}^m$ denoting  a collection of block partitions of $\{0,1,2,\dots,K\}$  witnessing that $\mathcal Q$ is a PIP. 
For $k=0,1,\dots,K$, define $E^{[k]}=\big{\{}v_{i,j}: \min A_{i,j} \leq k\big{\}}$. We  show that when $E^{[k]}$ is blue, then the remaining white vertices in $E^{[k+1]}$ can be forced blue  with forces along the paths using the standard zero forcing color change rule. This  implies that $\{E^{[k]}\}_{k=0}^K$ is an expansion sequence of $B=E^{[0]}=\{v_{i,0}\}_{i=1}^m$ for some relaxed chronology of forces $\mathcal{F}$ where all forces occur along the paths in $\mathcal Q$.

Suppose the vertices in $E^{[k]}=\bigl\{v_{i,j}: \min A_{i,j} \leq k\bigr\}$ are blue, and consider $v_{a,b} \in E^{[k+1]} \setminus E^{[k]}$. So 
 $Q_{a} $ is the unique element of $\mathcal Q$ containing ${v_{a,b}}$ and  there is a unique vertex $v_{a,b-1}$ preceding  $v_{a,b}$ in $Q_{a}$. Observe that  $\max A_{a,b-1}=k$.  Then  $v_{a,b-1}$ is blue, and  $v_{a,b}$ is the only white neighbor of  $v_{a,b-1}$ in $Q_{a}$. 
Also, since $\{(A_{i,j})_{j=0}^{n_i-1}\}_{i=1}^m$ is a collection of block partitions of $\{0,1,2,\dots,K\}$  witnessing that $\mathcal Q$ is a PIP, if  $v_{a,b-1}$ is adjacent to a vertex $v_{a',b'}\notin Q_{a}$, then $A_{a,b-1} \cap A_{a',b'} \neq \emptyset$.  Then $v_{ a',b'}\in E^{[k]}$ because $\max A_{ a,b-1}=k$. Thus ${ v_{a,b}}$ is the unique neighbor of $v_{a,b-1}$ in $G$ that is not blue, and ${v_{a,b-1}\to v_{a,b}}\in S(G,E^{[k]})$.  Hence, the remaining white vertices in $E^{[k+1]}$ can be colored blue by $E^{[k]}$ performing forces along the paths in $\mathcal Q$.  Thus we can select $F^{(k)}\subseteq S(G,E^{[k]})$ so that $\mathcal{F}=\{F^{(k)}\}_{k=0}^K$ is a relaxed chronology of standard forces. By construction, $\{(A_{i,j})_{j=0}^{n_i-1}\}_{i=1}^m$ records precisely the active time-steps of the corresponding vertices  $\{v_{i,j}\}_{j=0}^{n_i-1}\,_{i=1}^m$.

We now prove the reverse direction. Let $\mathcal C  =\{C_i\}_{i=1}^m$ be the chain set given by the relaxed chronology of forces $\mathcal F=\{F^{(k)}\}_{k=1}^{K}$ acting on the standard zero forcing set $B=\{v_{i,0}\}_{i=1}^m$ of $G$. For each forcing chain $C_i$, label the vertices $\{v_{i,j}\}_{j=0}^{n_i-1}$ so that $v_{i,j}\to v_{i,j+1}\in\bigcup_{k=1}^K F^{(k)}$, and define $\{A_{i,j}\}_{j=0}^{n_i-1}$ by $A_{i,j}=\act_{\mathcal{F}}(v_{i,j})$. By construction, $(A_{i,j})_{j=0}^{n_i-1}$ is a block partition of $\{0,1,2,\dots,K\}$ for all $i$. 
Now suppose there exists $v_{i_1,j_1}v_{i_2,j_2} \in E(G)$ such that $i_1 \not = i_2$ and  $A_{i_1,j_1} \cap A_{i_2,j_2} = \emptyset$.  Each $A_{i,j}$ is a collection of consecutive integers, so without loss of generality we assume that $A_{i_1,j_1}<A_{i_2,j_2}$.  However, this implies $K \not \in A_{i_1,j_1}$, so there exists $k\in \{1,2,\dots,K\}$ and a vertex $v_{i_1,j_1+1}$ such that $v_{i_1,j_1}\to v_{i_1,j_1+1}$ during time-step $k$ of $\mathcal F$.  Then $k-1 \in A_{i_1,j_1}$, so $k-1<y$ for all $y \in A_{i_2,j_2}$. In particular, $v_{i_2,j_2}$ is white when $v_{i_1,j_1}$ forces {$v_{i_1,j_1+1}$}.  Since $v_{i_2,j_2} \not \in C_{i_1}$, we know $v_{i_2,j_2} \ne v_{i_1,j_1+1}$. We conclude that when $v_{i_1,j_1}$ forces $v_{i_1,j_1+1}$, it has two white neighbors, which is a contradiction. Thus, if $v_{i_1,j_1}v_{i_2,j_2} \in E(G)$ with $i_1 \neq i_2$, then $A_{i_1,j_1} \cap A_{i_2,j_2} \not = \emptyset$, and $\mathcal C$ is a parallel increasing path cover. 
\end{proof}

\begin{corollary}\label{corparallel}
For any graph $G$, $\PIP(G)=\Z(G)$. 
\end{corollary}

It is known that given a graph $G$ and a zero forcing set $B$ of $G$, a chronological list of forces acting on $B$ will create a chain set of $G$, which is a partition of the vertices of $G$.  However, given a graph $G$, a parallel increasing path cover $\mathcal Q$ of $G$, and a collection of block partitions $\{\mathcal A_i\}_{i=1}^m$ that is a witness for $\mathcal Q$, Definition \ref{def:PIP} shows that the elements of the block partitions have a direct relationship with the edge set of $G$: 
\[\text{For $i_1,i_2$ distinct, }v_{i_1,j_1}v_{i_2,j_2} \in E(G) \Longrightarrow A_{i_1,j_1} \cap A_{i_2,j_2} \not = \emptyset.\]
Due to this, not only can one find parallel increasing path covers of a graph $G$, but one can also reverse this process. Given a
number $K \in \mathbb N$ and a collection of block partitions $\{\mathcal A_i\}_{i=1}^m$ of $\{0,1,2,\dots,K\}$, one can construct the family of graphs for which $\{\mathcal A_i\}_{i=1}^m$ outlines a relaxed chronology of forces.  The specifics of this fact are laid out in the next definition and corollary.

\begin{definition}\label{pipgraphs}
Let $K,m \in \mathbb N$ and for each $i \in \{1,2,\dots,m\}$ let $\mathcal A_i=(A_{i,j})_{j=0}^{n_i-1}$ be a block partition of $\{0,1,2,\dots,K\}$.   Let $V=\{v_{i,j}\}_{j=0}^{n_i-1}\,_{i=1}^m$, $E_1=\big{\{}v_{i,j_1}v_{i,j_2}: \abs{j_1-j_2}=1\big{\}}$, and \[ E_2 = \big{\{}v_{i_1,j_1}v_{i_2,j_2}: i_1 \not = i_2 \text{~and~} A_{i_1,j_1} \cap A_{i_2,j_2} \not = \emptyset\big{\}}.\]  Define $\mathcal G_{\{\mathcal A_i\}_{i=1}^m}$
to be the set of  all graphs of the form $G(E)=(V,E_1\cup E)$  where  $E\subseteq E_2$.
We refer to $\mathcal G_{\{\mathcal A_i\}_{i=1}^m}$ 
as the family of graphs induced by $\{\mathcal A_i\}_{i=1}^m$.  
\end{definition}

\begin{example}
Consider the collection of block partitions $\{\mathcal{A}_1,\mathcal{A}_2,\mathcal{A}_3\}$ given by
\begin{center}
\begin{tabular}{c c c c c c }
    $\mathcal{A}_1$: & $A_{1,0}=\{0\}$ & $A_{1,1}=\{1,2,3\}$ & $A_{1,2}=\{4\}$\\
    $\mathcal{A}_2$:& $A_{2,0}=\{0,1,2,3,4\}$ \\
    $\mathcal{A}_3$: & $A_{3,0}=\{0\}$ & $A_{3,1}=\{1\}$ & $A_{3,2}=\{2\}$ & $A_{3,3}=\{3\}$ & $A_{3,4}=\{4\}$.
\end{tabular}
\end{center}
Now define the set of vertices $V$ and sets of edges $E_1$ and $E_2$: 
\[V=\{v_{1,0},v_{1,1},v_{1,2},v_{2,0},v_{3,0},v_{3,1},v_{3,2},v_{3,3},v_{3,4}\},\]
\[E_1=\{v_{1,0}v_{1,1},v_{1,1}v_{1,2},v_{3,0}v_{3,1},v_{3,1}v_{3,2},v_{3,2}v_{3,3},v_{3,3}v_{3,4}\},\]
\begin{equation*}\begin{split}
E_2 = \{&v_{1,0}v_{2,0},v_{2,0}v_{3,0},v_{1,0}v_{3,0},v_{1,1}v_{2,0},v_{2,0}v_{3,1},v_{1,1}v_{3,1},\\
  &v_{2,0}v_{3,2},v_{1,1}v_{3,2},v_{2,0}v_{3,3},v_{1,1}v_{3,3},v_{1,2}v_{2,0},v_{2,0}v_{3,4},v_{1,2}v_{3,4}\}.\end{split}\end{equation*}
Then $G(E_2)$ is the graph  with the largest edge set among graphs in $\mathcal G_{\{\mathcal A_i\}_{i=1}^3}$, and is shown in  Figure \ref{ex:partitiontograph}.  
The graph $G(\emptyset)$ has the smallest edge set among graphs in $\mathcal G_{\{\mathcal A_i\}_{i=1}^3}$.  In Figure \ref{ex:partitiontograph}, the thick edges are the edges of $G(\emptyset)$.
If we choose any $H$ that is a subgraph of $G(E_2)$ and contains $G(\emptyset)$, then we obtain another graph in $  \mathcal G_{\{\mathcal A_i\}_{i=1}^3}$ such that $\{\mathcal{A}_1,\mathcal{A}_2,\mathcal{A}_3\}$ defines a collection of block partitions for the parallel increasing path cover induced by the edges in ${E_1}$.
\end{example}
\begin{figure}[h]
  \centering
\begin{tikzpicture}[scale=1.25]
\begin{scope}[every node/.style=vertex,bold edges]
\draw[label position=above]
    (0,3) node["$v_{1,0}$"] (v10) {} --
    (1,3) node["$v_{1,1}$"] (v11) {} --
    (4,3) node["$v_{1,2}$"] (v12) {};
\draw[label position=left]
    (0,2) node["$v_{2,0}$"] (v20) {};
\draw[label position=below]
    (0,1) node["$v_{3,0}$"] (v30) {} --
    (1,1) node["$v_{3,1}$"] (v31) {} --
    (2,1) node["$v_{3,2}$"] (v32) {} --
    (3,1) node["$v_{3,3}$"] (v33) {} --
    (4,1) node["$v_{3,4}$"] (v34) {};
\end{scope}
\draw[plain edges] (v20.center) -- (v30.center) to[bend left=90,looseness=1.2] (v10.center) -- (v20.center) -- (v31.center) -- (v11.center) -- (v20.center) -- (v32.center) -- (v11.center) -- (v33.center) -- (v20.center) -- (v12.center) -- (v34.center) -- (v20.center);
\node at (-1,3) {$Q_1$};
\node at (-1,2) {$Q_2$};
\node at (-1,1) {$Q_3$};
\end{tikzpicture}
  \caption{The graph $G(E_2)$, with the edges in $E_1$ shown in bold.}
  \label{ex:partitiontograph}
\end{figure}
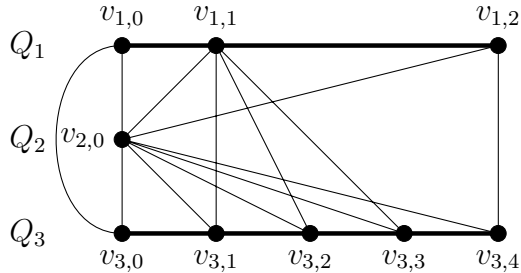

\begin{corollary}\label{corpipgraphs}
Let $G$ be a graph.
\begin{enumerate}
    \item If $G \in \mathcal G_{\{\mathcal A_i\}_{i=1}^m}$ for some collection of block partitions $\{\mathcal A_i\}_{i=1}^m$ with $\mathcal A_i=(A_{i,j})_{j=0}^{n_i-1}$, then $\mathcal C=\{C_i\}_{i=1}^m$ with $C_i=G\big{[}\{v_{i,j}\}_{j=0}^{n_i-1}\big{]}$ forms a chain set induced by a relaxed chronology of forces $\mathcal F$ acting on the zero forcing set $\{v_{i,0}\}_{i=1}^m$ of $G$.
    \item $\Z(G)=\min\{m: G \in \mathcal G_{\{\mathcal A_i\}_{i=1}^m} \text{ for some collection of block partitions } \{\mathcal A_i\}_{i=1}^m\}$.
\end{enumerate}
\end{corollary}

\begin{proof}
The proof follows immediately from Theorem \ref{parallel} and the definition of a parallel increasing path cover.
\end{proof}

\begin{corollary}\label{corpipprop}
Let $G$ be a graph, $B$ be a standard zero forcing set of $G$, $\mathcal F$ be a propagating family of forces of $B$ on $G$,  and let $\{\mathcal A_i\}_{i=1}^{\abs{B}}= \{(A_{i,j})_{j=0}^{n_i-1}\}_{i=1}^{|B|}$ be a collection of block partitions  that record the $\mathcal F$-active time-steps of the vertices  $\{v_{i,j}\}_{j=0}^{n_i-1}\, _{i=1}^{\abs{B}}$ of the chain set $\mathcal C=\{C_i\}_{i=1}^{\abs{B}}$  defined by $\mathcal F$. Then for each $H \in \mathcal G_{\{\mathcal A_i\}_{i=1}^{\abs{B}}}$,
\[\Z(H) \leq \abs{B} \hspace{0.25in} \text{ and } \hspace{0.25in} \pt(H, B) \leq  \pt(G,B).\]
If in addition, $B$ is an efficient standard zero forcing set of $G$, then
\[\Z(H) \leq \Z(G) \hspace{0.25in} \text{ and } \hspace{0.25in} \pt(H, B) \leq \pt(G).\]
\end{corollary}
 \begin{proof}
  Since $H \in \mathcal G_{\{\mathcal A_i\}_{i=1}^{\abs{B}}}$, it follows that $B$ is a zero forcing set of $H$ and $\mathcal F$ is a relaxed chronology of forces of $B$ on $H$.  Thus $\Z(H)\le \abs{B}$ and $ \pt(H,B) \leq \ct(\mathcal F)=\pt(G,B)$ because $\mathcal F$ is a propagating family of forces.   If in addition $B$ is an  efficient zero forcing set, then $B$ is a minimum zero forcing set and  $\Z(H) \leq \Z(G)$. Furthermore,  $\pt(G,B)= \pt(G)$. 
\end{proof}

\begin{remark}
For readers familiar with the minor monotone floor of zero forcing \cite{paramlong, HLS22}, we note that one can modify Definition \ref{def:PIP} (and necessary prior definitions) to define a structure analogous to PIPs for which corresponding versions of Theorem \ref{parallel} and Corollary \ref{corparallel} hold. Chain sets for the minor monotone floor of zero forcing in a graph $G$ correspond to chain sets for standard zero forcing in some supergraph of $G$ on the same vertex set.
Thus, to create an analog of Definition \ref{def:PIP} for the minor monotone floor of zero forcing, we simply allow each $\mathcal{Q}$ to be a path cover of some supergraph of $G$ on the same vertex set. Observe that property (\ref{non-empty-int}) is unchanged, as this is not related to $\mathcal{Q}$ being a path cover. The corresponding change for Definition \ref{pipgraphs} is to consider graphs of the form $G(E)=(V,E)$ with $E\subseteq E_1\cup E_2$.
\end{remark} 


\section{Applications of parallel increasing path covers }\label{s:PIP-apps}

We now apply the results of the previous section to standard and positive semidefinite zero forcing. Notice that in a parallel increasing path cover, the relationship between the collections of block partitions of time-steps and the edge set is based on nonempty intersections as described in \eqref{non-empty-int}. This property is preserved under certain operations, which we can use to generalize known results and establish new ones.

Hogben et al.\ defined the \textit{terminus} and \textit{reversal} of a set of forces in \cite{proptime}. They then compared the propagation time for a reversal with the original set of forces. We adapt these definitions to relaxed chronologies of forces.

\begin{definition}
Let $\mathcal{F}=\{F^{(k)}\}_{k=1}^N$ be a relaxed chronology of forces for a standard zero forcing set $B$ of a graph $G$.
\begin{enumerate}
    \item The \emph{terminus} of $\mathcal{F}$, denoted $\term(\mathcal{F})$, is the set of vertices of $G$ that do not perform any forces in $\mathcal{F}$. 
    \item The \emph{reversal} of $\mathcal{F}$, denoted $\rev(\mathcal{F})$, is the result of reversing the forces and time-steps in $\mathcal{F}$, i.e., $\rev(\mathcal{F})=\{F_{\rev}^{(k)}\}_{k=1}^K$ with
    \[F_{\rev}^{(k)}=\{v\to u:u\to v\in F^{(K-k+1)}\}.\]
\end{enumerate}
\end{definition}

Observe that the terminus depends only on the set of forces in  the relaxed chronology.  The next definition is equivalent to the definition of the standard propagation time of a set of standard forces in \cite{proptime} and extends the definition to other types of zero forcing.
\begin{definition}\label{d:pt-set-forces}
Let $G$ be a graph and $B$ be an $X$-forcing set of $G$.     For a set of $X$-forces $\mathcal{F}$ of $B$, define $E^{(0)}=B$ and for $t \ge 0$,  define $E^{(t+1)}$ to be the set of vertices  $w$ such that  $v\to w \in \mathcal F \cap S_X(G,\bigcup_{i=0}^{t} E^{(i)})$ for some $v$.  The {\em $X$-propagation  time of $\mathcal{F}$} in $G$, denoted  $\pt_X(G,\mathcal{F})$,   is the least $t_0$ such that $V(G)=\bigcup_{t=0}^{t_0} E^{(t)}$.  For a relaxed chronology $\mathcal F$, the $X$-\emph{propagation time} of $\mathcal{F}$ in $G$, denoted $\pt_X(G,\mathcal{F})$, is the propagation time of the underlying set of forces in $\mathcal{F}$.  
\end{definition}

It is clear that $\pt_X(G,\mathcal{F})\leq \ct(\mathcal{F})$.   The next lemma restates \cite[Theorem 2.6]{param}, generalizes  \cite[Observation 2.4]{proptime} to relaxed chronologies in  the obvious manner, and  improves \cite[Corollary 2.10]{proptime}.

 \begin{lemma}\label{lem:reverse}
Let $G$ be a graph, $B$ be a standard zero forcing set of $G$, and $\mathcal{F}$ be a relaxed chronology of forces for $B$. Then
\begin{enumerate}
    \item {\rm \cite{param}} $\term(\mathcal{F})$ is a standard zero forcing set of $G$,
    \item\label{cL43-2} $\rev(\mathcal{F})$ is a relaxed chronology for $\term(\mathcal{F})$, and 
    \item\label{cL43-3} $\pt(G,\rev(\mathcal{F}))= \pt(G,\mathcal{F})$. 
\end{enumerate}
\end{lemma}

\begin{proof} For \eqref{cL43-2},
let $\mathcal{F}=\{F^{(k)}\}_{k=1}^K$. By Theorem \ref{parallel}, $\mathcal{F}$ induces a parallel increasing path cover $\mathcal{Q}$  with corresponding collection of block partitions $\{(A_{i,j})_{j=0}^{n_i-1}\}_{i=1}^m$ recording active time-steps. Define the collection of block partitions $\{(R_{i,j})_{j=0}^{n_i-1}\}_{i=1}^m$ by
\[N \in A_{i,j} \Longleftrightarrow K-N \in R_{i,n_i-j-1}.\]  
Then $\{(R_{i,j})_{j=0}^{n_i-1}\}_{i=1}^m$ is a collection of block partitions of $\{0,1,2,\dots,K\}$ that is also a witness that $\mathcal Q$ is a parallel increasing path cover of $G$. The vertices of $\term(\mathcal{F})$ correspond precisely to $\{R_{i,0}\}_{i=1}^m$, so Theorem \ref{parallel}  implies $\rev(\mathcal{F})$ is a relaxed chronology of forces for $\term(\mathcal{F})$ with $\ct(\rev(\mathcal{F}))=\ct(\mathcal{F})$. 

For \eqref{cL43-3}, note that although  \cite[Corollary 2.10]{proptime} is stated for sets of forces of minimum zero forcing sets, it is a consequence of Lemma 2.9, which does not assume minimality.  Thus  $\pt(G,\rev(\mathcal{F}))\le  \pt(G,\mathcal{F})$.  Equality follows by noting that $\rev(\rev(\mathcal{F}))=\mathcal{F}$.  \end{proof}

\begin{remark}
Note that relaxed chronologies can give a simpler proof that  $\pt(G,\rev(\mathcal{F}))\le  \pt(G,\mathcal{F})$ than that given in Lemma 2.9 and Corollary 2.10 in \cite{param} (the reader is invited to consult the proof of \cite[Lemma 2.9]{param}, where a new parameter $Q_t(\F)$ is introduced for a set $\F$ of forces).  Let $\mathcal{F}'$ be the relaxed chronology obtained by propagating the forces in $\mathcal{F}$. Since $\ct(\rev(\F'))=\ct(\F')$ and $\mathcal{F}$ and $\mathcal{F}'$ have the same underlying set of forces (as do $\rev(\mathcal{F})$ and $\rev(\mathcal{F}')$), we  see that
\[\pt(G,\rev(\mathcal{F}))=\pt(G,\rev(\mathcal{F}'))\leq \ct(\rev(\mathcal{F}'))= \ct(\mathcal{F}')=\pt(G,\mathcal{F}')=\pt(G,\mathcal{F}).\]
\end{remark}

We can take further  advantage of the structure of parallel increasing path covers and block partitions to obtain additional results by focusing on when vertices are active.  Of particular note, we utilize this additional information provided by PIPs to compare each of PSD propagation time and power propagation time to standard propagation time.  Specifically, we are able to show that if $G$ is a graph and $m\in \mathbb N$ such that $m\geq \Z(G)$, then
\[\pt_+(G,m) \leq \left \lceil \frac{\pt(G,m)}{2} \right \rceil \text{ and }\ppt(G,m) \leq \left \lceil \frac{\pt(G,m)}{2} \right \rceil\]

(see Theorems \ref{t:pt-ub-4-ptp} and \ref{powerproptime}).

\begin{definition}\label{def:timesteps}
Let $\mathcal{F}=\{F^{(k)}\}_{k=1}^K$ be a relaxed chronology of forces for a graph $G$ and some standard zero forcing set. For any $N\in \{0,1,\dots,K\}$, define the following partition of $V(G)$:
\bea
        V_{\mathcal{F}}^{N-}  &=&\{v\in V(G)\mid\act_{\mathcal{F}}(v)\subseteq \{0,1,\dots,N-1\}\},\\
V_{\mathcal{F}}^{N}  &=&\{v\in V(G)\mid N\in \act_{\mathcal{F}}(v)\}, \text{ and}\\
V_{\mathcal{F}}^{N+}  &=&\{v\in V(G)\mid\act_{\mathcal{F}}(v)\subseteq \{N+1,N+2,\dots,K\}\}.
\eea
Now for any $0\leq M\leq N \leq K$, define 
\begin{eqnarray*}
\begin{aligned}
V_{\mathcal{F}}^{[M,N]} & =(V_{\mathcal{F}}^{M}\cup V_{\mathcal{F}}^{M+})\cap (V_{\mathcal{F}}^{N-}\cup V_{\mathcal{F}}^{N}). \\
\end{aligned}
\end{eqnarray*}
\end{definition}
Note that $V_{\mathcal{F}}^{N}$ consists of the vertices active after time-step $N$, while $V_{\mathcal{F}}^{N-}$ and $V_{\mathcal{F}}^{N+}$ partition the remaining vertices into ones that have performed a force during some time-step $N'<N$ and those who do not become blue until some time-step $N'>N$, respectively. Likewise, $V^{[M,N]}_{\mathcal F}$ consists of the vertices which are active at some time-step $k$ with $M \leq k \leq N$.

\begin{lemma}\label{lemma:truncate}
Let $\mathcal{F}=\{F^{(k)}\}_{k=1}^K$ be a relaxed chronology of forces for a graph $G$ and some  standard zero forcing set $B$, and let $N\in \{0,1,\dots,K\}$.
\begin{enumerate}
    \item\label{461} $B$ is a zero forcing set of $H=G[V_{\mathcal{F}}^{[0,N]}]$ with $\pt(H,B)\leq N$. 
    \item\label{462} $B'=V_{\mathcal{F}}^{N}$ is a zero forcing set of $H'=G[V_{\mathcal{F}}^{[N,K]}]$ with $\pt(H',B')\leq K-N$.
\end{enumerate}
\end{lemma}
\begin{proof}
By Theorem \ref{parallel}, $\mathcal{F}$ induces a PIP $\mathcal{Q}$ with witness $\{(A_{i,j})_{j=0}^{n_i-1}\}_{i=1}^m$, where $A_{i,j}=\act_{\mathcal{F}}(v_{i,j})$. Denote the vertices in $V_{\mathcal{F}}^N$ as $\{v_{i,j_i}\}_{i=1}^m$, and note that these vertices correspond to the members of the block partitions $A_{i,j_i}$ containing $N$. 

For \eqref{461}, form a path cover $\mathcal{Q}'$ of $H$ by restricting $\mathcal{Q}$ to the vertices $\{v_{i,j}\}_{j=0}^{j_i}\,{ }_{i=1}^m$. The collection of block partitions $\{(A_{i,j}')_{j=0}^{j_i}\}_{i=1}^m$ formed by starting with $\{(A_{i,j})_{j=0}^{j_i}\}_{i=1}^m$ and removing all elements in $\{N+1,N+2,\dots,K\}$ produces a collection of block partitions of $\{0,1,\dots,N\}$. Notice that $A_{i_1,j_1}'\cap A_{i_2,j_2}'\neq \emptyset$ precisely when $A_{i_1,j_1}\cap A_{i_2,j_2}\neq \emptyset$. Hence, $\{(A_{i,j}')_{j=0}^{j_i}\}_{i=1}^m$ is a witness that $\mathcal{Q}'$ is a parallel increasing path cover of $H$. The conclusions in (1) then follow from Theorem \ref{parallel}.

For \eqref{462}, form $\mathcal{Q}'$ by restricting $\mathcal{Q}$ to the vertices $\{v_{i,j}\}_{j=j_i}^{n_i-1}\,{}_{i=1}^m$ 
in $H'$. Now, start with $\{(A_{i,j})_{j=j_i}^{n_i-1}\}_{i=1,}^m$, remove all elements in $\{0,1,\dots,N-1\}$, and subtract $N$ from all remaining elements in each partition to obtain a collection of block partitions of $\{0,1,\dots,K-N\}$. By similar reasoning as given above, this is a witness that $\mathcal{Q}'$ is a parallel increasing path cover of $H'$, and Theorem \ref{parallel} implies (2).
\end{proof}

\begin{corollary}\label{cor:interval}
Let $\mathcal{F}=\{F^{(k)}\}_{k=1}^K$ be a relaxed chronology of forces for a graph $G$ for some standard zero forcing set. For any $0\leq M<N \leq K$, $V_{\mathcal{F}}^{M}$ and $V_{\mathcal{F}}^{N}$ are zero forcing sets of $G[V_{\mathcal{F}}^{[M,N]}]$, and both sets have propagation time at most $N-M$.
\end{corollary}
\begin{proof}
The results for $V_{\mathcal{F}}^{M}$ follow from first applying part (1) of Lemma \ref{lemma:truncate} with $N$ and then applying part (2) with $M$. Lemma \ref{lem:reverse} then implies the corresponding results for $V_{\mathcal{F}}^{N}$, as this is $\term(\mathcal{F}')$ for the relaxed chronology $\mathcal{F}'=\{F^{(k+M)}\}_{k=1}^{N-M}$  on the graph $G[V_{\mathcal{F}}^{[M,N]}]$ with zero forcing set $V_{\mathcal{F}}^M$.
\end{proof}

Using these techniques we also obtain results for positive semidefinite zero forcing and propagation time.

\begin{lemma}\label{lemma:cutset}
Let $\mathcal{F}=\{F^{(k)}\}_{k=1}^K$ be a relaxed chronology of forces for a graph $G$ and some  standard zero forcing set. Then no vertices in $V_{\mathcal{F}}^{N-}$ are adjacent to vertices in $V_{\mathcal{F}}^{N+}$. In particular, if both sets are nonempty, then $V_{\mathcal{F}}^{N}$ is a vertex cut of $G$.
\end{lemma}
\begin{proof}
Using Theorem \ref{parallel}, $\mathcal{F}$ induces a parallel increasing path cover $\mathcal{Q}$ with corresponding collection of block partitions $\{(A_{i,j})_{j=0}^{n_i-1}\}_{i=1}^m$. Consider $v_{i_1,j_1}\in V_{\mathcal{F}}^{N-}$ and $v_{i_2,j_2}\in V_{\mathcal{F}}^{N+}$ with corresponding sets of active time-steps $A_{i_1,j_1}$ and $A_{i_2,j_2}$. By Definition \ref{def:PIP}, if $v_{i_1,j_1} v_{i_2,j_2}\in E(G)$ for $i_1 \not = i_2$, then it must be that $A_{i_1,j_1} \cap A_{i_2,j_2} \not = \emptyset$, and if $v_{i_1,j_1}v_{i_1,j_2}\in E(G)$ with $i_1=i_2$ and $j_1 < j_2$, it must be that $\max(A_{i_1,j_1})+1=\min(A_{i_1,j_2})$. Neither of these situations can occur since $A_{i_1,j_1}\subseteq \{0,1,\dots,N-1\}$ and $A_{i_2,j_2}\subseteq \{N+1,N+2,\dots,K\}$.
\end{proof}

\begin{theorem}\label{t:pt-ub-4-ptp}
Let $G$ be a graph and $m\in \mathbb N$ such that $m\geq \Z(G)$. Then
\[\pt_+(G,m) \leq \left \lceil \frac{\pt(G,m)}{2} \right \rceil.\]
\end{theorem}

\begin{proof}
Let $B$ be a standard zero forcing set of $G$ of size $m \geq \Z(G)$ such that $\pt(G,B)=\pt(G,m)$, with corresponding propagating set of forces $\mathcal{F}=\{F^{(k)}\}_{k=1}^{\pt(G,m)}$ (where $F^{(k)}$ denotes the set of forces during time-step $k$). Using Theorem \ref{parallel}, this induces a PIP $\mathcal{Q}=\{Q_i\}_{i=1}^m$, where each $Q_i=\{v_{i,j}\}_{j=0}^{n_i-1}$ corresponds to a block partition $(A_{i,j})_{j=1}^{n_i-1}$ of $\{0,1,2,\dots ,\pt(G,m)\}$. Now set $N=\left \lceil \frac{\pt(G,m)}{2} \right \rceil$ and let $B'=V_{\mathcal{F}}^N$.  By the definition of the PSD color change rule and Lemma \ref{lemma:cutset}, PSD forcing in $H_1=G[V_{\mathcal{F}}^{0,N}]$ and $H_2=G[V_{\mathcal{F}}^{N,K}]$ with $B'$ blue will occur independently and simultaneously. Using Corollary \ref{cor:interval} on $H_1$ and $H_2$, we conclude that \[\pt_+(G,B')\leq \max\{\pt(H_1,B),\pt(H_2,B)\}\leq \max\{\pt(G,B)-N,N\}=\left \lceil \frac{\pt(G,m)}{2} \right \rceil.\qedhere\]
\end{proof}

\begin{remark}
In the preceding proof, $B'$ is chosen to be $V_{\mathcal{F}}^N$ for $N=\left\lceil\frac{\pt(G,m)}{2}\right\rceil$, but any choice of $N$ such that $0 \leq N \leq \pt(G,m)$ will yield a PSD forcing set of $G$ of size $m$ with propagation time bounded above by $\max\{\pt(G,m)-N,N\}$. In particular, when $\pt(G,m)$ is odd, the choice of $N=\left\lfloor\frac{\pt(G,m)}{2}\right\rfloor$ also establishes the bound in Theorem \ref{t:pt-ub-4-ptp}.  

\end{remark}

\begin{corollary}\label{throt}
For any graph $G$,
\[\thr_+(G) \leq \min_{m \geq \Z(G)} \Big{(}m + \left \lceil \frac{\pt(G,m)}{2} \right \rceil \Big{)}.\]
\end{corollary}

\begin{corollary}\label{Z=Z_+}
For any graph $G$ such that $\Z_+(G)=\Z(G)$,
\[\pt_+(G) \leq \left\lceil\frac{\pt(G)}{2}\right\rceil.\]
\end{corollary}

\begin{remark}In \cite{warnberg}, it was shown that for $t \in \{1,2,3\}$, $\Z_+(P_s \cartProd P_t)=\Z(P_s \cartProd P_t)=\min\{s,t\}$, where $P_s \cartProd P_t$ denotes the Cartesian product of the path graphs $P_s$ and $P_t$, i.e., the graph with vertex set $V(P_s)\times V(P_t)$ such that $(u,v)$ is adjacent to $(u',v')$ if and only if (1) $u=u'$ and $vv'\in E(P_t)$, or (2) $v=v'$ and $uu'\in E(P_s)$.  It was also shown that $\pt(P_s \cartProd P_t)=\max\{s,t\}-1$, and $\pt_+(P_s \cartProd P_t)=\left\lceil\frac{\max\{s,t\}-1}{2}\right\rceil$.  Thus the infinite class of graphs $P_s \cartProd P_t$ establishes that the bound in Corollary \ref{Z=Z_+} is sharp.
\end{remark}   

Note that in the proof of Theorem \ref{t:pt-ub-4-ptp}, we started with a relaxed chronology $\mathcal{F}=\{F^{(k)}\}_{k=1}^K$ and chose a single time-step of vertices $V_{\mathcal{F}}^N$ as the PSD forcing set. One can generalize this approach and choose multiple time-steps. For example, choosing $B=V_{\mathcal{F}}^{\lceil K/4\rceil}\cup V_{\mathcal{F}}^{\lfloor 3K/4\rfloor}$ would construct a PSD forcing set with $\pt_+(G,B)\leq  \left\lceil \frac{K}{4}\right\rceil$. In general, this technique multiplies the size of the zero forcing set by some positive integer $\ell$, and reduces propagation time by approximately a factor of $2\ell$.

Letting $\ppt(G,m)$ denote the power propagation time for sets of size $m$, we can also establish power domination versions of the results above. While one can do this by directly generalizing the techniques of this section, we will see that the results of the next section simplify arguments significantly. Hence, we present the proof of the following theorem in Appendix \ref{appendix:power}.

\begin{theorem}\label{powerproptime}
Let $G$ be a graph and $m\in \mathbb{N}$ such that $m\geq \Z(G)$. Then \[\ppt(G,m)\leq \lc \frac{\pt(G,m)}{2} \rc.\]
\end{theorem}


\section{Path bundles}\label{s:relaxed-chron-subgraphs}

In this section, we introduce path bundles (see Definition \ref{d:pathbundle}). These are 
  collections of paths contained in PSD forcing trees that are PIPs for the induced subgraph on the vertices in the bundle. We apply path bundles to compare PSD and standard propagation times, and we establish a PSD analog of the reversal of a standard zero forcing process. We start by showing that relaxed chronologies provide us a convenient method of restricting forces to subgraphs.

\begin{definition}
Let $G$ be a graph with induced subgraph $H$, and let $X$-CCR be a color change rule. 
\begin{enumerate}
    \item Given a set of $X$-forces $F$ between the vertices in $G$, define its \emph{restriction} to $H$, denoted $F|_H$, to be $\{v\to u\in F: u,v\in V(H)\}$.
    \item Given a relaxed chronology $\mathcal{F}=\{F^{(k)}\}_{k=1}^K$ for some $X$-forcing set $B$ of $G$, define its \emph{restriction} to $H$, denoted $\mathcal{F}|_H$, to be $\{F^{(k)}|_H\}_{k=1}^K$.
\end{enumerate}
\end{definition}

Notice that in $\mathcal{F}|_H$, we preserve the time-step of each force.  If $\mathcal{F}|_H$ is a relaxed chronology for some zero 
forcing set $B'$ of $H$, we have that $\ct(\mathcal{F}|_H)=\ct(\mathcal{F})$. Additionally, it is possible that for some $k$, $F^{(k)}\neq \emptyset$ and $F^{(k)}|_H=\emptyset$.

\begin{definition}
Let $G$ be a graph, $H$ be an induced subgraph of $G$, and $B$ be a PSD forcing set of $G$.  Let $\mathcal F$ be a relaxed chronology of PSD forces of $B$ on $G$ with expansion sequence $\{E^{[k]}\}_{k=0}^K$ and PSD forcing trees $\mathcal T=\{T_i\}_{i=1}^{|B|}$.  A component $T$ of $T_i\cap H$ is a \emph{forcing subtree} of $T_i$ in $H$, and $u\in V(T)$ is an \emph{initial vertex} if either $u \in B$ or there exists $k\in \mathbb{N}$ such that $u \in E_{\mathcal F}^{(k)}$ and $v \rightarrow u \in F^{(k)}$ but $v \not \in V(T)$.  When a forcing subtree $T$ is a path, we also call $T$ a \emph{forcing subpath}. 
\end{definition}

\begin{lemma}\label{subgraph}
Let $G$ be a graph, $H$ be an induced subgraph of $G$, $B$ be a PSD forcing set of $G$, and $\mathcal{F}$ be a relaxed chronology of PSD forces of $B$ on $G$ with forcing trees $\mathcal T$. Let $B'$ be the set of initial vertices in the forcing subtrees of $\mathcal{T}$ in $H$. Then $B'$ is a PSD forcing set of $H$, and $\mathcal{F}|_H$ defines a relaxed chronology of PSD forces for $B'$ in $H$ with $\ptp(H,\mathcal{F}|_H)\leq \ptp(G,\mathcal{F})$.
\end{lemma}

\begin{proof}
Let $\mathcal{F}=\{F^{(k)}\}_{k=1}^K$ with corresponding expansion sequence $\{E_{\mathcal{F}}^{[k]}\}_{k=0}^K$. We first show that the PSD forces in $F^{(k)}|_H$ are valid when $E_{\mathcal{F}}^{[k-1]}\cap V(H)$ is blue. Consider $v\to u\in F^{(k)}|_H$. Since $v\to u\in F^{(k)}$, $u$ is the unique neighbor of $v$ in some component of $G-E_{\mathcal{F}}^{[k-1]}$. The components of $H-E_{\mathcal{F}}^{[k-1]}$ are formed from induced subgraphs of $G-E_{\mathcal{F}}^{[k-1]}$, and hence $v\to u$ is a valid PSD force in $H$ during time-step $k$.

We now claim that if $(E_{\mathcal{F}}^{[k-1]}\cap V(H))\cup B'$ is blue, then all vertices in $(E_{\mathcal{F}}^{[k]}\cap V(H))\cup B'$ will be blue after the forces in $F^{(k)}|_H$ are performed. Consider $v\to u\in F^{(k)}$ with $u\in (E_{\mathcal{F}}^{[k]} \cap V(H))\cup B'$. If $v\notin V(H)$, then $u$ is an initial vertex in some forcing subtree. By definition, $u\in B'$, so we conclude that $u$ is initially blue, and hence also blue after time-step $k$. Otherwise, $v\in V(H)$, so the preceding paragraph implies that $v\to u\in F^{(k)}|_H$ is a valid PSD force in $H$, and $u$ will be blue after time-step $k$. Combining these two results with induction on $k$, we conclude that $\mathcal{F}|_H$ is a relaxed chronology of PSD forces for $B'$ in $H$ with expansion sequence $\{(E_{\mathcal{F}}^{[k]}\cap V(H))\cup B'\}_{k=0}^K$. 
\end{proof}

In the case that $\mathcal F$ is a propagating set of PSD or standard forces, then by Lemmas \ref{subgraph} and  \ref{subgraph2}, it follows that 
\[\pt_+(H,B') \leq \pt_+(G,B) \text{ and } \pt(H,B') \leq \pt(G,B).\]  The  proof of Lemma \ref{subgraph} can be adapted to establish the next lemma.

\begin{lemma}\label{subgraph2}
Let $G$ be a graph, $H$ be an induced subgraph of $G$, $B$ be a standard zero forcing set of $G$, and $\mathcal{F}$ be a relaxed chronology of standard forces of $B$ on $G$ with chain set $\mathcal{C}$. Let $B'$ be the set of initial vertices in the forcing subpaths of $\mathcal{C}$ in $H$. Then $B'$ is a standard zero forcing set of $H$, and $\mathcal{F}|_H$ defines a relaxed chronology of standard forces for $B'$ in $H$ with $\pt(H,\mathcal{F}|_H)\leq \pt(G,\mathcal{F})$.
\end{lemma}

We now turn our attention to the case when the restriction of PSD forcing results in standard zero forcing and the resulting zero forcing set has the same size as the original PSD forcing set.

\begin{definition}\label{d:pathbundle}
Let $G$ be a graph, $B$ be a PSD forcing set of $G$, and let $\mathcal{F}$ be a relaxed chronology of PSD forces of $B$ on $G$ with associated PSD forcing tree cover $\mathcal{T}=\{T_i\}_{i=1}^{|B|}$. Let $\mathcal{Q}=\{Q_i\}_{i=1}^{|B|}$ be a collection of paths such that  $V(Q_i)\subseteq V(T_i)$.  Define $H=G[\bigcup_{Q_i\in \mathcal{Q}}V(Q_i)]$. We say that $\mathcal{Q}$ is a \emph{path bundle} of $\mathcal{F}$ if $\mathcal{F}|_H$ is a relaxed chronology of standard forces in $H$ for the initial vertices in $\mathcal{Q}$. We abuse notation and use $\mathcal{F}|_{\mathcal{Q}}$ to also denote $\mathcal{F}|_H$.
\end{definition}

Observe that by definition, the chain set of $\mathcal{F}|_{\mathcal{Q}}$ will be $\mathcal{Q}$ itself. Additionally, any path bundle $\mathcal{Q}$ is a parallel increasing path cover of the induced subgraph $H$. When $\mathcal{Q}$ contains the vertices of $B$ and $\abs{\mathcal Q}=\abs{B}$, then the set of initial vertices of $\mathcal{F}$ in $H$ is precisely $B$, and $\pt(H,B)\leq \ptp(G,B)$; see Corollary \ref{cor:bundlept}. There are trivial examples of path bundles, such as the paths consisting of just the vertices of $B$ or the collection consisting of a single path $Q_1\subseteq T_1$. However, there are also many nontrivial cases where  $\mathcal{F}|_{\mathcal{Q}}$ provides us information about $\mathcal{F}$ and the original graph $G$. Since $\F|_{\mathcal{Q}}$ is a relaxed chronology of standard  forces for the subgraph $H$ induced by the vertices of the paths in $\mathcal{Q}$, we can consider its terminus  $\term(\mathcal{F}|_{\mathcal{Q}})$. 

\begin{definition}
Let $G$ be a graph, $B$ be a PSD forcing set of $G$, and $\mathcal{F}$ be a relaxed chronology of PSD forces of $B$ on $G$. Fix a vertex $x\in V(G)$.  For $k=0,1,2,\dots,\rd(x)-1$, define $C^{k}_x=\comp(G-E_{\mathcal F}^{[k]},x)$ to be the component of $G-E_{\mathcal F}^{[k]}$ containing $x$.  Construct  collections of paths $\mathcal{Q}^{[k]}$ as follows:
\begin{enumerate}
    \item Let $\{v^{0}_i\}_{i=1}^{|B|}$ be the vertices of $B$, and let $\mathcal{Q}^{[0]}=\{Q_i^{[0]}\}_{i=1}^{|B|}=\{\{v_i^{0}\}\}_{i=1}^{|B|}$ be the collection of single-vertex paths on the vertices of $B$.
    \item  For each $i$, if $v^{k}_i$ forces some vertex  $w\in C^{k}_x$ at time-step $k+1$ of $\mathcal{F}$ (i.e., $v^k_i\to w\in F^{(k+1)}$), then define $v^{k+1}_i=w$, and construct $Q_i^{[k+1]}$ by adding $v^{k+1}_i$ to the end of the path $Q_i^{[k]}$. For the remaining paths $Q_i^{[k]}$, define $Q_i^{[k+1]}=Q_i^{[k]}$ and $v^{k+1}_i=v^{k}_i$.
    \item Finally, let $\mathcal{Q}^{[k+1]}=\{Q_i^{[k+1]}\}_{i=1}^{|B|}$.
\end{enumerate}
We call $\mathcal{Q}=\mathcal{Q}^{[\rd(x)]}$ the \emph{path bundle of $\mathcal{F}$ induced by $x$}.  If we wish to speak of this type of path bundle in general rather than a specific instance of one, then we will refer to them as {\em vertex-induced path bundles} or {\em path bundles induced by a vertex}.
\end{definition}

\begin{lemma} \label{path_builder}
Let $G$ be a graph, $B$ be a PSD forcing set of $G$, $\mathcal F$ be a relaxed chronology of PSD forces of $B$ on $G$, and $\mathcal T$ be the PSD forcing tree cover defined  by $\mathcal F$.  Then at each time-step $k$, given a component $C$ of $G - E_{\mathcal F}^{[k]}$ and a PSD forcing tree $T \in \mathcal T$, there is at most one vertex $v \in V(T) \cap E_{\mathcal F}^{[k]}$ such that $v$ has white neighbors in $C$. 
\end{lemma}

\begin{proof}
We proceed by contradiction. Suppose there are distinct $v,w\in V(T)\cap E^{[k]}_{\mathcal F}$ respectively adjacent to $v',w'\in V(C)$. Then at time $k$, there exists a path of white vertices  $v',\dots, w'$ entirely within $C$. 
Let $u\in T$ be the vertex such that the forces $u\to v_1\to \dots \to v$ and $u\to w_1\to \dots \to w$ are in $\mathcal{F}$ with $v_1\neq w_1$. Without loss of generality, assume $v_1$ is forced at the same time as or before $w_1$.  Letting $j$ be the time when $v_1$ is forced, we see that the path $v_1,\dots, v,v',\dots, w',w, \dots, w_1$ consists entirely of white vertices in a single component of $G-E_{\mathcal{F}}^{[j-1]}$. Then at time $j$, vertex $u$ is adjacent to two white vertices $v_1$ and $w_1$ in the same component of $G-E_{\mathcal{F}}^{[j-1]}$, which contradicts the fact that $u\to v_1$ at this time-step.
\end{proof}

\begin{lemma}\label{inducedbundle}
Let $G$ be a graph, $B$ be a PSD forcing set of $G$, and $\mathcal{F}$ be a relaxed chronology of PSD forces of $B$ on $G$. For any $x\in V(G)$, the path bundle $\mathcal{Q}$ of $\mathcal{F}$ induced by $x$ is a path bundle that contains $B$ and $x$. Furthermore, if $\mathcal{F}$ is a propagating family of PSD forces, then in $\mathcal{F}|_{\mathcal{Q}}$, there is at least one force in each time-step until all vertices of $G[\bigcup_{P\in\mathcal{Q}}V(P)]$ are blue.
\end{lemma}

\begin{proof}
We prove by induction that each $\mathcal{Q}^{[k]}$ is a path bundle. Label $B=\{v_{i}^0\}_{i=1}^m$, and note that $\mathcal{Q}^{[0]}$ is the set of trivial paths on $B$, so this is a path bundle. 
 Now suppose that we have constructed the path bundle $\mathcal{Q}^{[k]}=(Q_1^{[k]},\dots,Q_{m}^{[k]})$ with corresponding induced subgraph $H_{k}$. If no forces occur into $C^{k}_x$ at time-step $k+1$ of $\mathcal F$, then $\mathcal{Q}^{[k+1]}=\mathcal{Q}^{[k]}$ is again a path bundle.  Otherwise, we let $S\subseteq \{1,2,\dots,m\}$ contain all indices such that for each $i\in S$, there exists $v_i\in T_i\cap E_{\mathcal F}^{[k]}$ and $w_i\in C^{k}_x$ with $v_i\to w_i\in F^{(k+1)}$, where $\mathcal T=\{T_i\}_{i=1}^m$ is the PSD forcing tree cover induced by $\mathcal F$. 

If $v_i=v_i^0$, then $v_i^0\in Q_i^{[k]}$, and by definition of the PSD forcing rule, $v_i\to w_i$ is a valid standard force in $G[V(H_{k})\cup V(C^{k}_x)]$ with the vertices $v_i$ and $w_i$ forming a forcing chain, and thus a path, of length 1. If $v_i\neq v_{i}^0$, then there exists some sequence of forces in $\mathcal{F}$ from $v_i^0$ to $v_i$ consisting of vertices in $\mathcal{T}_i$. Some vertex $v$ in this sequence is the last one in $Q_i^{[k]}$, as we have that the vertex $v_i^0\in Q_i^{[k]}$. Lemma \ref{path_builder} implies that this $v$ was the unique vertex in $T_i$ with a white neighbor in $C^{k'}_x$ for some $k' \leq k$. Since $C^{k}_x$ is contained in $C^{k'}_x$, our construction of $\mathcal{Q}^{[k]}$ implies that if $k'<k$ and $v$ forced a vertex $v'$ in $C^{k'}_x$ during $F^{(k'+1)}$, then $v'$ must also be in $Q_i^{[k]}$. Since we chose $v$ as the last vertex in $Q_i^{[k]}$ in the chain from $v_i^0$ to $v_i$, we conclude that $v$ has not forced any vertex in $C^{k'}_x$, and in particular $v=v_i$. Again, we see that $v_i\to w_i$ is a valid standard force in $G[V(H_{k})\cup V(C^{k}_x)]$. Hence, if we define $Q_i^{[k+1]}=Q_i^{[k]}$ when $i\notin S$ and $Q_i^{[k+1]}$ to be $Q_i^{[k]}$ with $w_i$ appended when $i\in S$, then $(v_i^0,\dots,v_i,w_i)$ is a forcing chain, and thus a path, in $T_i[V(H_{k+1})]$, and we obtain a strictly larger path bundle $\mathcal{Q}^{[k+1]}$ from $\mathcal{Q}^{[k]}$ by adding the vertices in $C^{k}_x$ forced during time-step $k+1$. 

This process terminates at $\rd(x)$, and induction implies that $\mathcal{Q}^{[\rd(x)]}$ is a path bundle. By construction, it contains both $B$ and $x$. Additionally, if $\mathcal{F}$ is propagating, then there cannot be time-steps where no forces occur in $C^{k}_x$, as PSD forcing in each component occurs independently. Hence, $\mathcal{Q}^{[k]}\neq \mathcal{Q}^{[k+1]}$ for all $0\leq k\leq \rd(x)-1$, and at least one force occurs at each time-step of $\mathcal{F}|_{\mathcal{Q}}$ until we reach $\rd(x)$.
\end{proof}

\begin{corollary}\label{cor:bundlept}
Let $G$ be a graph, $B$ be a PSD forcing set of $G$ of size $k$, and $\mathcal F$ be a propagating family of PSD forces for $B$ on $G$.  Let  $\mathcal H$ be the collection of subgraphs $H$ of $G$ such that $H=G\left[\bigcup_{P \in \mathcal Q} V(P) \right]$ for some path bundle $\mathcal Q$ of $\mathcal F$ induced by some $x \in V(G)$.
Then
\[\max_{H \in \mathcal H}\pt(H,B) \leq \pt_+(G,B).\]
Moreover, if $x\in  B^{\left(\pt_+(G,B)\right)}$, then
\[\pt_+(G,B) \leq \min_{H \in \mathcal H}\left\vert V(H) \right\vert - k.\]
\end{corollary}

\begin{proof}
Since $\mathcal F$ is a propagating family of PSD forces for $B$ on $G$, we have that
\[\pt(H,B)\leq \ct(\mathcal{F}|_{\mathcal{Q}}) =\ptp(G,\mathcal{F})=\ptp(G,B),\]
and the first result follows from maximizing on the left.
Using Lemma \ref{inducedbundle}, any such $H$ with corresponding $B$ and $\mathcal{F}$ satisfies
\[\ptp(G,\mathcal{F}) \leq |V(H)|-k.\]
Since $\ptp(G,\mathcal{F})=\ptp(G,B)$, the second result follows from minimizing on the right.
\end{proof}

If $\mathcal{F}$ is a relaxed chronology of forces for a PSD forcing set $B$, then the set of vertices that do not perform a force in $\mathcal{F}$ need not have the same size as $B$. However, vertex-induced path bundles allow us to produce PSD forcing sets of the same size using the terminus of the resulting standard zero forcing set.

\begin{theorem}\label{reverse}
Let $G$ be a graph and $B$ be a PSD forcing set of $G$. Let $\mathcal{F}$ be a relaxed chronology of PSD forces for $B$ on $G$, and let $\mathcal{Q}$ be the path bundle of $\mathcal{F}$ induced by $x\in V(G)$. Then $\term(\mathcal{F}|_{\mathcal{Q}})$ is a PSD forcing set of $G$. Furthermore, a relaxed chronology of PSD forces for $\term(\mathcal{F}|_{\mathcal{Q}})$ can be constructed by reversing the forces between vertices in $\mathcal{Q}$ and preserving all remaining forces. 
\end{theorem}

\begin{proof}
Since $\mathcal Q$ is a path bundle, $\mathcal{F}|_{\mathcal{Q}}$ is a relaxed chronology of standard forces.  By Theorem \ref{parallel}, 
one can use $\mathcal{F}|_{\mathcal{Q}}$ to construct a collection of block partitions witnessing that $\mathcal Q$ is a parallel increasing path cover of $H=G[\bigcup_{Q_i\in \mathcal{Q}}V(Q_i)]$ and thus the collections of vertices $\{V_{\mathcal F|_\mathcal{Q}}^N\}_{N=0}^{\rd(x)}$ as described in Definition \ref{def:timesteps}.  Define the collection of vertices $\{R^k\}_{k=0}^{\rd(x)}$ such that for each $k$, $R^k=\bigcup_{i=0}^kV_{\mathcal F|_{\mathcal Q}}^{\rd(x)-i}=V^{[\rd(x)-k,\rd(x)]}_{\mathcal F|_{\mathcal Q}}$.  First note that $\term(\mathcal F|_{\mathcal Q})=R^0$. We will now show inductively that for $k$ with $0 \leq k < \rd(x)$, if $R^k$ is blue, then $R^{k+1}$ can be forced blue using only the reverses of PSD forces contained in $\mathcal F|_{\mathcal Q}$.  This inductive process terminates at $R^{\rd(x)}$, which contains the PSD forcing set $B$  of $G$.  Thus any remaining white vertices can be forced using only PSD forces found in $\mathcal F$, and this will show that $\term(\mathcal F|_{\mathcal Q})$ is a PSD forcing set of $G$.

Let $k\in \{0,1,\dots,\rd(x)-1\}$, and suppose $R^k$ is currently blue.  Since $\mathcal F|_{\mathcal Q}$ is a relaxed chronology of standard forces in $H$, it follows by Lemma \ref{lem:reverse} that $\rev(\mathcal F|_{\mathcal Q})$ is a relaxed chronology of standard forces for $\term(\mathcal{F}|_{\mathcal{Q}})$ in $H$.  Due to this, for each $u \in R^{k+1} \setminus R^k$, there exists a vertex $v \in R^k$ such that $u$ is the only white neighbor of $v$ in $H$. We claim that $u$ is the only white neighbor of $v$ in $\comp(G-R^k,u)$, the component of $G-R^k$ that contains $u$, and hence $v\to u\in S_+(G,R^k)$.

Let $k'=\rd(x)-k$.  Since $u\to v$ at time $k'$, $v\in V(C^{k'-1}_x)\subseteq V(C^{t-1}_x)$ for any $t\leq k'$. Now, suppose by way of contradiction that there exists $u'\in (N_G(v)\cap V(\comp(G-R^k,u)))\setminus V(H)$.  Then there exists a path $u'=p_0,p_1,\dots ,p_m=u$ with $p_i\notin R^k$.  Let $j=\min\{i\colon p_i\in V(H)\}>0$, which is well-defined since $u\in V(H)$. Finally, let $y=p_j$.

It is asserted that for all $0\leq i<j$, $\rd_{\mathcal{F}}(p_i)\geq k'$.  Otherwise, when the first such $p_i$ was forced by some vertex $z$ at time $t<k'$, it was in the same component of $G-E^{[t-1]}_{\mathcal{F}}$ as $v$.  Thus, $p_i\in V(\comp(G-E^{[t-1]}_{\mathcal{F}},v))=V(C^{t-1}_x)$.  This would imply, though, that $z\to p_i\in\mathcal{F}|_{\mathcal{Q}}$, and so $p_i\in V(H)$.

Thus, for all $0\leq i<j$, $p_i\in V(\comp(G-E^{[k'-1]}_{\mathcal{F}},v))=V(C^{k'-1}_x)$.  Since $y\in V(H)$, but $y\notin R^k$, there must be some time-step $t\leq k'$ and some vertex $y'\in V(C^{t-1}_x)$ such that $y\to y'\in\mathcal{F}|_{\mathcal{Q}}$ at time $t$.  However, this implies that $\{p_{j-1},y'\}\subseteq N_G(y)\setminus E^{[t-1]}_{\mathcal{F}}$, so $y\to y'$ is not a valid PSD force at time $t$, a contradiction.
\end{proof}

\begin{remark}
    Alternatively, it is worth noting that the process of multiple-vertex migration introduced in \cite{PSDprop1} by Hogben et al.\ could also be used inductively for the purposes of the proof.  In some sense the restriction to the induced subgraph given by vertex-induced path bundles allows the concept of parallel increasing path covers to be useful in the PSD forcing setting. Together, parallel increasing path covers and vertex-induced path bundles do globally in the graph $G$ what multiple-vertex migration does locally.
\end{remark}

\begin{corollary}\label{c:trade-v-in-T}
Let $G$ be a graph with $\mathcal{T}$ being the PSD forcing trees for some PSD forcing set of size $k$. For any $v\in V(G)$, there exist a PSD forcing set $B$ of size $k$ containing $v$ and a relaxed chronology of PSD forces for $B$ with $\mathcal{T}$ as its induced forcing trees.
\end{corollary}

\begin{proof}
Let $\mathcal{F}$ be a relaxed chronology of forces inducing $\mathcal{T}$, and let $\mathcal{Q}$ be the path bundle of $\mathcal{F}$ induced by $v$. Theorem \ref{reverse} implies that $\term(\mathcal{F}|_{\mathcal{Q}})$ is a PSD forcing set of $G$ containing $v$, and a corresponding relaxed chronology $\mathcal{F}'$ can be constructed by reversing the forces between vertices in $\mathcal{Q}$ and preserving the other forces. Observe that $\mathcal{F}$ and $\mathcal{F}'$ contain forces between the same pairs of vertices, albeit possibly in different directions. Hence, they induce the same forcing trees $\mathcal{T}$.
\end{proof}

Our results on path bundles have connections with \emph{rigid linkages} studied by Ferrero et al. \cite{rigid}. We detail the explicit relationship in Appendix \ref{appendix:linkage}.


\color{black}

\section{Concluding remarks}

In this paper, we introduced the concept of a relaxed chronology as a generalization of a chronological list of forces.  This framework permitted the development of parallel increasing path covers (PIPs), an alternative formulation of a standard zero forcing process from a more global perspective.  These PIPs allowed for the construction of families of graphs with predetermined chain sets, as well as the establishment of certain relationships between standard zero forcing and other zero forcing variants.

Relaxed chronologies also were used to discuss the restriction of forcing to subgraphs.  A special case of such restrictions, called path bundles, was introduced for which the restriction of PSD forcing is standard zero forcing.  These path bundles allowed the construction of PSD forcing sets containing a chosen vertex that have the same PSD forcing trees as a chosen initial PSD forcing set.  Connections between path bundles and rigid linkage forcing were established in the appendices.

Looking forward, it is hoped that these new frameworks and concepts will prove useful outside the confines of this particular paper.  It is intended that these ideas and definitions will help serve as a foundation for a broader and more robust discussion of zero forcing and its variants.

\appendix

\section{Proof of Theorem \ref{powerproptime}}\label{appendix:power}

In this section, we generalize the results in Section \ref{s:PIP-apps} to establish that \[\ppt(G,m)\leq \lc \frac{\pt(G,m)}{2}\rc.\] Note that we use the results of Section \ref{s:relaxed-chron-subgraphs} extensively, further demonstrating their usefulness. We start with generalizing the sets from Definition \ref{def:timesteps}.

\begin{definition}
Let $\mathcal{F}=\{F^{(k)}\}_{k=1}^K$ be a relaxed chronology of standard forces for a graph $G$ and some standard zero forcing set. For any $M,N\in \{0,1,\dots,K\}$ with $M\leq N$, define
\begin{eqnarray*}\label{eq:ppt}
\begin{aligned}
V_{\mathcal F}^{[M,N)} & =V_{\mathcal F}^{[M,N]}\setminus V_{\mathcal F}^{N},\\
V_{\mathcal F}^{(M,N]} & =V_{\mathcal F}^{[M,N]}\setminus V_{\mathcal F}^{M}, \text{ and}\\
V_{\mathcal F}^{(M,N)} & =V_{\mathcal F}^{[M,N]}\setminus (V_{\mathcal F}^{M} \cup V_{\mathcal F}^{N}).
\end{aligned}
\end{eqnarray*}
Finally, define $V_{\mathcal{F}}^{\bd(M+)}$ to be the initial vertices of the chain set for $\mathcal F$ in $G[V_{\mathcal F}^{(M,K]}]$ and $V_{\mathcal{F}}^{\bd(N-)}$ to be the initial vertices for $\rev(\mathcal F)$ in $G[V_{\mathcal F}^{[0,N)}]$.
\end{definition}

\begin{lemma}\label{lem:openintervals}
Let $\mathcal F=\{F^{(k)}\}_{k=1}^K$ be a relaxed chronology of standard forces for a graph $G$ and some standard zero forcing set $B$ and let $M \in \{0, 1,2,\dots,K-1\}$ and $N \in \{1,2,3,\dots, K\}$.  Then 
\begin{enumerate}
    \item \label{504} $B_1=V_{\mathcal F}^{\bd(N-)}$ is a zero forcing set of $H_1=G[V_{\mathcal F}^{N-}]=G[V_{\mathcal F}^{[0,N)}]$ with $\pt(H_1,B_1) \leq N-1$.
    \item \label{506} $B_2=V_{\mathcal F}^{\bd(M+)}$ is a zero forcing set of $H_2=G[V_{\mathcal F}^{M+}]=G[V_{\mathcal F}^{(M,K]}]$ with $\pt(H_2,B_2) \leq K-M-1$.
\end{enumerate}
\end{lemma}

\begin{proof}
Lemma \ref{lemma:truncate} implies $V_{\mathcal{F}}^{0}$ is a zero forcing set of  $H=G[V_{\mathcal{F}}^{[0,N-1]}]$ with propagation time at most $N-1$. Lemma \ref{subgraph2} then implies $V_{\mathcal{F}}^{0}$ is a zero forcing set of the subgraph of $H$ given by $H_1=G[V_{\mathcal{F}}^{[0,N)}]$, where the propagation time is again at most $N-1$. Since $B_1$ is the terminus of $\mathcal{F}|_{H_1}$, Lemma \ref{lem:reverse} implies (\ref{504}). Applying (\ref{504}) on $\rev(\mathcal{F})$ and $\term(\mathcal{F})=V_{\mathcal{F}}^K$ then implies claim (\ref{506}).
\end{proof}

In order to simplify notation, for the following result we will be following the convention that given a graph $G$, a set of vertices $B \subseteq V(G)$, and a subgraph $H$ of $G$, if $B \cap V(H)$ is a zero forcing set of $H$, then we will say that $B$ is a zero forcing set of $H$.

\begin{corollary}
Let $\mathcal{F}=\{F^{(k)}\}_{k=1}^K$ be a relaxed chronology of forces for a graph $G$ for some standard zero forcing set. For any $0\leq M<N \leq K$,
\begin{itemize}
    \item $V_{\mathcal{F}}^{\bd(M+)}$ and $V_{\mathcal{F}}^{\bd(N-)}$ are zero forcing sets of $G[V_{\mathcal{F}}^{(M,N)}]$, and both sets have propagation time at most $N-M-2$.
    \item $V_{\mathcal{F}}^{M}$ and $V_{\mathcal{F}}^{\bd(N-)}$ are zero forcing sets of $G[V_{\mathcal{F}}^{[M,N)}]$, and both sets have propagation time at most $N-M-1$.
    \item $V_{\mathcal{F}}^{\bd(M+)}$ and $V_{\mathcal{F}}^{N}$ are zero forcing sets of $G[V_{\mathcal{F}}^{(M,N]}]$, and both sets have propagation time at most $N-M-1$.
\end{itemize}
\end{corollary}
\begin{proof}
Each claim follows immediately from applying Lemma \ref{subgraph2} with either Corollary \ref{cor:interval} or the preceding lemma.
\end{proof}

We now use these results to establish our power propagation time bound.

\begin{proof}[Proof of Theorem \ref{powerproptime}]
We can assume $\pt(G,m)>0$, as otherwise the result is trivial. As in the proof of Theorem \ref{t:pt-ub-4-ptp}, we let $B$ be an $m$-efficient standard zero forcing set of $G$ with corresponding relaxed chronology $\mathcal{F}=\{F^{(k)}\}_{k=1}^{\pt(G,m)}$. Letting $N=\lc \frac{\pt(G,m)}{2}\rc$, we have that $N_G[V_{\mathcal{F}}^N]$ contains $V_{\mathcal F}^{\bd(N-)}$ and $V_{\mathcal F}^{\bd(N+)}$. By Lemma \ref{lem:openintervals}, $V_{\mathcal F}^{\bd(N-)}$ is a standard zero forcing set of $G[V_{\mathcal F}^{N-}]$ with $\pt\left(G[V_{\mathcal F}^{N-}],V_{\mathcal F}^{\bd(N-)}\right) \leq N-1$. Likewise, $V_{\mathcal F}^{\bd(N+)}$ is a standard zero forcing set of $G[V_{\mathcal F}^{N+}]$ with $\pt\left(G[V_{\mathcal F}^{N+}],V_{\mathcal F}^{\bd(N+)}\right) \leq K-N-1 \leq N-1$, since $K=\pt(G,m)$.  Since Lemma \ref{lemma:cutset} implies that $V_{\mathcal{F}}^N$ separates $V_{\mathcal F}^{N-}$ and $V_{\mathcal F}^{N+}$ and every vertex in $V_{\mathcal{F}}^N$ is blue, the forcing process can proceed independently in $G[V_{\mathcal F}^{N-}]$ and $G[V_{\mathcal F}^{N+}]$.  It thus follows that
\[\ppt(G,m)\leq 1+\max\left\{\pt\left(G[V_{\mathcal F}^{N-}],V_{\mathcal F}^{\bd(N-)}\right), \pt\left(G[V_{\mathcal F}^{N+}],V_{\mathcal F}^{\bd(N+)}\right)\right\} \leq \lc \frac{\pt(G,m)}{2} \rc.\qedhere \]
\end{proof}

\begin{remark}
In the preceding proof, $V_{\mathcal{F}}^N$ for $N=\left\lceil\frac{\pt(G,m)}{2}\right\rceil$ is chosen as a power dominating set, but, as was the case with for PSD forcing, any choice of $N$ such that $0 \leq N \leq \pt(G,m)$ will yield a power dominating set of $G$ of size $m$ with propagation time bounded above by $\max\{\pt(G,m)-N,N\}$. In particular, when $\pt(G,m)$ is odd, the choice of $N=\left\lfloor\frac{\pt(G,m)}{2}\right\rfloor$ also establishes the bound in Theorem \ref{powerproptime}. 
\end{remark}

\section{Rigid linkages and path bundles}\label{appendix:linkage}
In this section, we discuss a connection between path bundles and rigid linkages. We establish that the path bundle induced by a relaxed chronology and a vertex forms a rigid linkage. We refer the reader to Ferrero et al.\ \cite{rigid} for further details of rigid linkages and their applications. 

Let $G$ be a graph. A {\em linkage} in $G$ is a subgraph whose connected components are paths. Note that a linkage need not contain all vertices of $G$. Let $\alpha,\beta \subseteq V(G)$.  A linkage $\mathcal P$ is an {\em $(\alpha,\beta)$-linkage} if $\alpha$ consists of one endpoint of each path in $\mathcal P$ and $\beta$ consists of the other endpoint of each path.  In the case that a path is a single vertex, the vertex is in both $\alpha$ and $\beta$.  A linkage $\mathcal P$ is {\em $(\alpha,\beta)$-rigid} if $\mathcal P$ is the unique $(\alpha,\beta)$-linkage in $G$.  A linkage $\mathcal P$ is {\em rigid} if there exist $\alpha$ and $\beta$ such that $\mathcal P$ is an $(\alpha,\beta)$-linkage and $\mathcal P$ is $(\alpha,\beta)$-rigid.

 For $X \subseteq V(G)$, define the boundary $\partial_G (X)$ of $X$ to be the set of vertices not in $X$ that have at least one neighbor in $X$.  When $C$ is a subgraph of $G$, define $\partial_G(C)=\partial_G(V(C))$. 
  Rigid linkage (RL) forcing is defined by the {rigid linkage color change rule (CCR-RL)}. 
  Given a current set of blue vertices $B^{[k]}$, an application of CCR-RL
  (to go from time-step $k$ to time-step $k+1$) 
  consists of the following:  
  \bit\item Choose a component $C$ of $G-B^{[k]}$, such that $\partial_G(C)$ does not contain any inactive blue vertices (that is, those which have previously performed a force).  
  \item Select an active blue vertex $u$ such that $w$ is the only white neighbor of $u$ in $C$:
  \bit\item Let $u$  force $w$, so that $B^{[k+1]}=B^{[k]}\cup \{w\}$.
  \item Update the active vertices ($w$ becomes active, $u$ becomes inactive).
  \eit
  \eit

For a given rigid linkage forcing process $\F$ on $G$ with $r$ time-steps:
\begin{itemize}
    \item An RL-forcing chain is a path $(v_0,v_1,\dots,v_{\ell})$ such that $v_0 \in B^{[0]}$, $v_i \rightarrow v_{i+1}$ for all $i=0,\dots,\ell -1$, and $v_\ell$ is active after time-step $r$.
    \item The RL-chain set is the set of all RL-forcing chains (for the given forcing process).
\end{itemize}

Ferrero et al.\ established that standard zero forcing is a special type of RL-forcing and identified the close connection between RL-forcing and rigid linkages.

\begin{proposition}\label{rigidzf}\cite[Proposition 2.3]{rigid}
Any standard zero forcing process on a graph $G$ is a rigid linkage forcing process.
\end{proposition}

\begin{theorem}\label{rl}\cite[Theorem 2.10]{rigid}
Let $G$ be a graph and $\mathcal P$ be a linkage in $G$.  Then $\mathcal P$ is a rigid linkage if and only if $\mathcal P$ is an RL-chain set under some RL-forcing process.
\end{theorem}

Given a graph $G$ and a set of blue vertices $B$, since PSD forcing works independently in different components of $G-B$, we have the observation below.

\begin{observation}\label{empty-rl}
Let $G$ be a graph, $B$ be a PSD forcing set of $G$, $\mathcal F$ be a chronological list of PSD forces of $B$ on $G$, $x \in V(G)$, $\mathcal Q$ be the path bundle of $\mathcal F$ induced by $x$, and $H=G\left[\bigcup_{P \in \mathcal Q}V(P)\right]$.  Construct $\mathcal F^*$ by performing the forces in $\mathcal F|_{\mathcal Q}$ first (in order) and afterwards performing the remaining forces in $\mathcal F$ (in order). Then $\mathcal{F}^*$ is a chronological list of PSD forces for $G$.
\end{observation}  

We now establish the connection between rigid linkages and vertex-induced path bundles.

\begin{theorem}\label{PIP-rl}
Let $G$ be a graph, $B$ be a PSD forcing set of $G$, $\mathcal F$ be a relaxed chronology of PSD forces of $B$ on $G$, $x \in V(G)$, $\mathcal Q$ be the path bundle of $\mathcal F$ induced by $x$, and $H=G\left[\bigcup_{P \in \mathcal Q}V(P)\right]$.    Then $\mathcal Q$ is a rigid linkage of $G$.  
\end{theorem}

\begin{proof}
    Let $K=\abs{V(H)}-\abs{B}$.  By Observation \ref{empty-rl}, we can suppose without loss of generality that $\mathcal F$ is a chronological list of forces where the $K$ forces in $\mathcal F|_{\mathcal Q}$ are performed first.  Since $\mathcal Q$ is a linkage, we proceed by first showing that $\{F^{(k)}|_{\mathcal Q}\}_{k=1}^K$ is a valid RL-forcing process and then applying Theorem \ref{rl} to complete the proof. Prior to application of the first force, all vertices in $B$ are active, so the first force in $\mathcal{F}$ is a valid RL-force.  

    We proceed by induction. Fix $k\in \{0,1,\dots,K-1\}$, and suppose that $B^{[k]}$ is blue and every force in $\{F^{(i)}|_{\mathcal Q}\}_{i=1}^k$ is a valid RL-force.  Additionally, let $u_{k+1} \rightarrow v_{k+1} \in F^{(k+1)}|_{\mathcal Q}$. We assert that this force is a valid RL-force.  
Suppose $u_j \rightarrow v_j \in F^{(j)}|_{\mathcal Q}$ for some $j \leq k$.  Then $v_j$ was the only neighbor of $u_j$ in $C^{j-1}_x$.
Since $V(C^k_x) \subseteq V(C^{j-1}_x) \setminus B^{[k]}$ and $v_j \in B^{[k]}$, it follows that $u_j$ has no neighbors in $C^k_x$.  Thus $\partial_G(C^k_x)$ contains no inactive blue vertices. Since $v_{k+1}$ is the only white neighbor of $u_{k+1}$ in $C^k_x$, $u_{k+1} \rightarrow v_{k+1}$ is a valid RL-force. Finally, $\{F^{(k)}|_{\mathcal Q}\}_{k=1}^K$ is a valid RL-forcing process with RL-chain set $\mathcal Q$, so applying Theorem \ref{rl} completes the proof.
\end{proof}

\section*{Acknowledgements}
This research began at the American Mathematical Society Mathematics Research Community ``Finding Needles in Haystacks: Approaches to Inverse Problems using Combinatorics and Linear Algebra'' with support from the National Science Foundation, and the authors thank AMS and NSF.  
The research of all the authors was partially supported by NSF grant 1916439.  The research of Yaqi Zhang was also partially supported by Simons Foundation grant 355645 and NSF grant DMS 2000037.

\bibliographystyle{plain}

\end{document}